\documentclass[12pt]{amsart}
\usepackage{amssymb,amsmath,amsfonts,amsthm,bbm,mathrsfs,times,graphicx,color,comment,enumerate}
\usepackage[utf8]{inputenc}
\usepackage{tikz}
\usetikzlibrary{arrows,shapes}
\usetikzlibrary{fadings}
\usetikzlibrary{intersections}
\usetikzlibrary{decorations.pathreplacing}
\usetikzlibrary{quotes}
\usetikzlibrary{angles}
\usetikzlibrary{babel}
\usepackage{mathtools}

\usepackage{booktabs}

\newcommand\numberthis{\addtocounter{equation}{1}\tag{\theequation}}

\DeclareMathOperator{\re}{\mathbb{R}e}
\DeclareMathOperator{\im}{\mathbb{I}m}
\newcommand{\llnorm}[1]{|||#1|||}
\newcommand{\norm}[1]{\left\|#1\right\|}
\newcommand{\lnorm}[1]{ \left\| #1 \right\|}

\newcommand{\spr}[2]{\langle #1,#2 \rangle}

\renewcommand{\i}{\mathrm i}

\newcommand{\INF}{{\infty}}

\newcommand{\dist}{\mbox{dist}}
\newcommand{\tta}{\theta}

\newcommand{\OM}{\Omega}

\newcommand{\sph}{{{\mathbb S}^ 1}}

\newcommand{\del}{\partial}

\newcommand{\Gam}{\varGamma}
\newcommand{\ol}{\overline}
\newcommand{\ds}{\displaystyle}
\newcommand{\dba}{\overline{\partial}}

\newcommand{\BR}{\mathbb{R}}
\newcommand{\BC}{\mathbb{C}}

\newcommand{\BZ}{\mathbb{Z}}
\newcommand{\BN}{\mathbb{N}}

\newcommand{\bu}{{\bf u}}

\newcommand{\bv}{{\bf v}}
\newcommand{\bw}{{\bf w}}
\newcommand{\bc}{{\bf c}}
\newcommand{\ba}{{\bf a}}

\newcommand{\bg}{{\bf g}}

\newcommand{\bbf}{{\bf f}}

\newcommand{\bF}{{\bf F}}

\newcommand{\bzero}{\mathbf 0}

\newcommand{\btheta}{\boldsymbol \theta}
\newcommand{\balpha}{\boldsymbol \alpha}

\newcommand{\bbeta}{\boldsymbol \beta}

\newcommand{\B}{\mathcal{B}}

\newcommand{\BT}{\mathcal{T}}

\newtheorem{theorem}{Theorem}[section]
\newtheorem{prop}{Proposition}[section]
\newtheorem{lemma}{Lemma}[section]
\newtheorem{cor}{Corollary}[section]

\newtheorem{example}{Example}
\numberwithin{equation}{section}

\title[Inversion of the momenta ray transform ]
{Inversion of the attenuated momenta ray transform of planar symmetric tensors}
\begin{document}
	\date{\today}
\author{Hiroshi Fujiwara}	\address{Graduate School of Informatics,  Kyoto University, Yoshida Honmachi, Sakyo-ku, Kyoto 606-8501, Japan }	\email{fujiwara@acs.i.kyoto-u.ac.jp}
	
	\author{David Omogbhe}
	\address{Faculty of Mathematics, Computational Science Center, University of Vienna, Oskar-Morgenstern-Platz 1, 1090 Vienna, Austria}
	\email{david.omogbhe@univie.ac.at}
	\author{Kamran Sadiq}
	\address{Johann Radon Institute for Computational and Applied Mathematics (RICAM), Altenbergerstrasse 69, 4040 Linz, Austria}
	\email{kamran.sadiq@ricam.oeaw.ac.at} 
		
	\author{Alexandru Tamasan}
	\address{Department of Mathematics, University of Central Florida, Orlando, 32816 Florida, USA}
		\email{tamasan@math.ucf.edu}

		
	\maketitle
	\begin{abstract}
		We present a reconstruction method that stably recovers the real valued, symmetric tensors  compactly supported in the Euclidean plane, from knowledge of their attenuated momenta ray transform.  The problem is recast as an inverse boundary value problem for a system of transport equations, which we solve by an extension of Bukhgeim's $A$-analytic theory. The method of proof is constructive. To illustrate the reconstruction method, we present results obtained in the numerical implementation for the non-attenuated case of 1-tensors. This new version now includes the results of  the preprint arXiv: 2307.10758.
	\end{abstract}

	\section{Introduction} \label{sec:intro}

	We consider the problem of recovering a real valued, symmetric $m$-tensor field $\bbf$ compactly supported in the plane, $m\geq 1$, from knowledge of its $0,1,...,m$-th  attenuated  moment ray transforms. This problem is motivated by some engineering applications: for $m=1$ in Doppler tomography \cite{norton88,braunHauk,sparSLP95},
and Magneto-acousto-electrical tomography \cite{haideretal,kunyanskyetal23},  for $m=2$ in inverse kinematic problems in isotropic elastic media \cite{abbeyetal,hendriksetal} and for $m=4$ in anisotropic media  \cite{aben79, sharafutdinov_book94}. The non-attenuated case also arises in the linearization of the boundary rigidity problem \cite{sharafutdinov_book94,stefanovUhlmann98, stefanovUhlmannetal}.

	When the data is limited to the $0$-moment, the (non-attenuated) ray transform has a large kernel containing all the potential tensors $d \bg$ with $\bg$ vanishing at the boundary of the support, and a vast literature in tensor tomography concerns the recovery of the solenoidal part of the field, see \cite{sharafutdinov_book94, paternainSaloUhlmann14,  palamodov09, schuster08, holmanStefanov} 
	and references therein.

In order to recover the entire tensor, three types of additional data have been proposed: 
some longitudinal ray data as in \cite{sharafutdinov_book94}, 
some transverse  data as in
\cite{desaiLionheart16,louis, kunyanskyetal23}, or a mixture of the two types
\cite{derevtsov23}.
This work concerns the longitudinal data case: In \cite{sharafutdinov_book94} it is shown that the entire field is uniquely determined from the combined $k^{th}$-moment ray transform for $0\leq k\leq m$; for brevity we call it the {momenta} ray transform. Inversion of the momenta ray transform has been the subject of recent research interests: in the Euclidean setting some inversion formulas were given in \cite{derevtsovSvetov15,derevtsovetal21}, with reconstruction for the $m=1$ case in \cite{andersson, kunyanskyetal23}, and the recent sharp stability estimates in \cite{denisiuk}. In the non-Euclidean setting
the unique determination result was shown for simple real analytic Riemannian manifolds in \cite{abhishekMishra}, extended to simple Riemannian surface in \cite{venkeMishraMonard19}, 
with inversion for $m=1$ and sources on a curve in \cite{mishra}, and stability estimates in \cite{sharafutdinov17,venkeetal19}. Since we also consider the attenuated case, it should be mentioned that the $0$-moment  attenuated Doppler transform in the Euclidean plane
is known to uniquely determine the entire field in subdomains  where the attenuation is positive \cite{bal04, kazantsevBukhgeimJr07, tamasan07}; note, however, 
that those methods become unstable for small attenuation.

Different from the above referenced works, herein we give a reconstruction method that recovers the entire $\bbf$  without appealing to the Helmholtz decomposition; in bounded domains this is an advantage that avoids the ambiguity of a harmonic potential. 
Our approach considers a new inverse source problem for a weakly coupled system of transport equations, which we solve by an extension of Bukhgeim's theory of $A$-analyticity \cite{bukhgeimBook}. Stability estimates for the reconstruction require new higher order a priori estimates for solutions of the Bukhgeim-Beltrami equation. Of independent interest, these estimates introduce a new analytical tool; see Section \ref{sec:estimate_BBE}.

In the end we present the results obtained by applying the reconstruction method to three numerical examples in the  $m=1$ (Doppler) case.
 The analysis of the numerical algorithms involved  is subject to a separate discussion.

\section{Statement of the main result}\label{sec:Result}
In this section we introduce notation and state our main result. Let $\Omega$ be  the unit disc with boundary $\Gam$. For $m\geq 1$ fixed integer, let $\bbf= (f_{i_1i_2...i_m})$  be a real valued symmetric $m$-tensor supported in $\ol\OM$. Furthermore, for $s\geq 1$, we assume that $\bbf$ has components in the Sobolev space of functions of square integrable derivatives which,  up to order $s$, vanish at the boundary. We denote the space of such tensors by 
$\ds 
H^{s}_0(\mathbf{S}^m; \OM) =\left \{\bbf =(f_{i_1 \cdots i_m}) \in \mathbf{S}^m( \OM): f_{i_1 \cdots i_m}\in H^{s}_0(\OM)\right \}.
$
The symmetry refers to $f_{i_1i_2...i_m}$ being invariant under any transposition of the indexes $i_1,...,i_m\in\{1,2\}$. 

With the summation convention understood over repeated indexes, 
for $(x,\btheta)\in\BR^2\times \sph$ we denote by $\langle \bbf(x), \btheta^m \rangle = f_{i_1 \cdots i_m} (x) \theta^{i_1} \cdot \theta^{i_2} \cdots \theta^{i_m}\ds $
the action of $\bbf$ on $\ds \underbrace{ \btheta  \otimes \btheta \otimes \cdots \otimes \btheta}_m$.

As in \cite{denisiuk} (defined for $a=0$), in here we work with the  $k^{th}$-moment attenuated ray transform 
\begin{align}\label{Iaf} 
I_a^k\bbf(x,\btheta) &:=\int_{-\infty}^{\infty} t^k 
e^{-\int_t^\infty a( \Pi_{\btheta}(x)+s\btheta)ds} \langle \bbf( \Pi_{\btheta}(x)+t\btheta), \btheta^m \rangle dt, \quad  0\leq k \leq m, \end{align}
where $\Pi_{\btheta}(x)=x- (x\cdot\btheta)\btheta$ is the projection of $x$ onto $\btheta^\perp$.  Both the tensor $\bbf$ and the function $a$ in \eqref{Iaf} 
are assumed extended by zero outside $\OM$. The function $a$ in \eqref{Iaf} models an attenuation. In the non-attenuated case ($a= 0$), we use the notation $I^k\bbf := I_0^k\bbf$.
	Note that in the non- attenuated case the definition \eqref{Iaf} is slightly different than the original definition in \cite{sharafutdinov_book94}: therein for $k \geq 1$ the  $k^{th}$ moment ray transform is not constant along the lines (the derivative at $x$ in the direction of $\btheta$  does not vanish).

The following elementary result (see the Appendix \ref{sec:elementary_results} for a proof) reduces the inversion of the momenta ray transform to an inverse boundary value problem for a system of transport equations. Let $\Gam_\pm:=\{(x,\btheta)\in \del \OM \times\sph:\, \pm\nu(x)\cdot\btheta>0 \}$ be the incoming (-), respectively outgoing (+), unit tangent sub-bundles of the boundary; where  $\nu(x)$ is the outer unit normal at $x \in \del \OM$.
\begin{prop}\label{prop_transEq} Let $s\geq 0$ and $ m \geq 1$ be arbitrarily fixed, and let $\bbf \in H^{s}_0(\mathbf{S}^m; \OM)$ and   
	$a\in C^{s,\mu}(\ol\OM), \mu >1/2$. The system 
	\begin{subequations}\label{bvp_UK_transport}
		\begin{align}\label{TransportEq_u0}
			\btheta\cdot\nabla u^0(z,\btheta) +a (z) u^0(z,\btheta) &=  \langle  \bbf(z), \btheta^m \rangle, \quad \text{for } \, (z, \btheta) \in \ol \OM \times \sph,
			\\  \label{TransportEq_uk}
			\btheta\cdot\nabla u^k(z,\btheta) +  a (z) u^k(z,\btheta) &=  u^{k-1}(z,\btheta), \quad \text{for } \, 1 \leq k \leq m,
		\end{align}
		subject to
		\begin{align}\label{uk_Gam-}
			u^k \lvert_{\Gam_{-}} &=0, \quad 0 \leq k \leq m,
		\end{align}	 
	\end{subequations} has a unique solution $u^k \in H^{s}(\Omega\times\sph)$. In particular, if $s\geq 1$, then  $u^k \lvert_{\Gam \times \sph} \in H^{s}(\sph ; H^{s-\frac{1}{2}}(\Gam))$.
	
	Moreover,  $ \langle u^0\lvert_{\Gam_{+}} , u^1\lvert_{\Gam_{+}},   \cdots ,  u^m\lvert_{\Gam_{+}} \rangle$     are in a one-to-one correspondence with 	the attenuated momenta ray transform
	$\langle I_a^0\bbf, I_a^1\bbf,  \cdots , I_a^m\bbf \rangle$ in \eqref{Iaf} 
	via the relations
	\begin{equation}\label{eq:gk-Ik}
		\begin{aligned}
			u^0 \lvert_{\Gam_{+}}(x,\btheta) &=  I_a^0\bbf (x,\btheta),  &&\\ 
			u^k \lvert_{\Gam_{+}} (x, \btheta) &= \sum_{n=1}^{k} (-1)^{n-1} \frac{(x \cdot \btheta)^n}{n !} u^{k-n} \lvert_{\Gam_{+}}(x,\btheta) + \frac{(-1)^{k}}{k !} I^k_a\bbf (x, \btheta), &&\mbox {for }1\leq k \leq m. 
		\end{aligned}
	\end{equation}	
\end{prop}

The inverse boundary value problem considered here seeks to recover the solution of the system \eqref{TransportEq_u0} and \eqref{TransportEq_uk} together with the unknown source $\bbf$ from knowledge of $u^k|_{\Gam\times\sph}$,  for all $k=0,...,m$.

For specificity, we call $u^k(z, \btheta)$ of \eqref{bvp_UK_transport}
 \emph{the $k$-level flux}, for 
 $k=0,...,m$.

We use a Fourier approach, where functions $u$ on $\OM\times\sph$ are characterized by the sequence valued map of their Fourier coefficients $ u_{-n}(z)=\frac{1}{2\pi}\int_0^{2\pi} u(z,\btheta)e^{\i n\theta}d\theta\ds$ (non-positive indexes are sufficient) in the angular variable,  $$\OM\ni z\mapsto  \bu(z):=\langle u_{0}(z), u_{-1}(z),u_{-2}(z),\cdots\rangle,$$ 
and  work in the weighted $l^{2}$ spaces 
\begin{align}\label{weightedSpaces}
	l^{2,\frac{p}{2}}(\BN;H^q(\OM)) &: = \left \{\bu= \langle u_{0}, u_{-1},u_{-2},... \rangle: \lnorm{\bu}_{\frac{p}{2},q}^2 := \sum_{j=0}^{\INF} (1+j)^{p} \lnorm{ u_{-j} }^{2}_{ H^q(\OM)} < \INF \right \}
\end{align} with traces $\bg=\bu|_\Gam\in  l^{2,\frac{p}{2}}(\BN;H^{q-\frac{1}{2}}(\Gam))$. The first  (weight)  index $p$ refers to the smoothness in the angular variable, while the second index $q$ shows the smoothness in the spatial variable. 

Since $\bbf$ vanishes on the lines outside $\OM$, we restrict $I_a^k\bbf$ to the lines intersecting $\ol\OM$. These lines are parametrized by points on the boundary $\Gam$ and directions in $\sph$, yielding $I_a^k\bbf$ a function on the torus $\Gam\times\sph$.
While $\Gam$ is also the unit circle, we keep the notation to differentiate from the set of directions. 
The maps in $ l^{2,\frac{p}{2}}(\BN;H^{q-\frac{1}{2}}(\Gam))$  have a norm defined directly on the Fourier lattice
\begin{align}\label{unorm1}
	\lnorm{\bg}_{\frac{p}{2},q-\frac{1}{2}}^2 = 
	\sum_{j=0}^{\INF} \sum_{n=-\INF}^{\INF}  (1+j)^{p}  (1+|n|)^{2q-1} |g_{-j,n}|^2,
\end{align}where $  g_{-j,k}=\frac{1}{2\pi}\int_0^{2\pi} g_{-j}(e^{\i \beta})e^{-\i k\beta}d\beta\ds$, for $k\in\BZ$, $j\geq 0$.

We use the notation  
$\ds
\lnorm{\bv}\lesssim \lnorm{\bw},
$ whenever $\ds \lnorm{\bv}\leq C \lnorm{\bw}$ for some constant $C>0$ independent of $\bv$ and $\bw$, and denote $\ds \lnorm{\bv}\approx \lnorm{\bw}$ if $\ds \lnorm{\bv}\lesssim \lnorm{\bw}\lesssim \lnorm{\bv}$.

	\begin{theorem}\label{Th_main}
	Let $\OM$ be the unit disc, $m \geq 1$ be an integer, and $a\in C^{m+1,\mu}(\ol\OM), \mu >1/2$. For some unknown real valued  $m$-tensor $\bbf \in H^{m+\frac{3}{2}}_0(\mathbf{S}^m; \OM)$, let $ \mathbb{I}_a\bbf :=\langle I_a^0\bbf, I_a^1\bbf, I_a^2\bbf,  \cdots , I_a^m\bbf \rangle$ be the attenuated momenta ray transform as in \eqref{Iaf}.
Then $I_a^k\bbf \in  H^{m+\frac{3}{2}}(\sph; H^{m+\frac{1}{2}}(\Gam))$ for $ 0\leq k \leq m$, and $\bbf $ is determined by  
	$ \mathbb{I}_a\bbf$ with the estimate 
	\begin{align}\label{stability_f_mEven1}
		\lnorm{ \bbf}_{L^2(\OM)}^2
		\lesssim 	\sum_{j=0}^{m}	\lnorm{  I_a^{j} \bbf}^2_{m+\frac{3}{2},j+\frac{1}{2}}.
	\end{align}  
The method of proof is constructive. 
\end{theorem}


\section{A priori estimates for solutions of inhomogeneous Bukhgeim-Beltrami equation} 
\label{sec:estimate_BBE}

The stability estimate in Theorem \ref{Th_main}
requires a priori estimates  
for solutions of the inhomogeneous Bukhgeim-Beltrami equation 
\begin{align}\label{BukhBeltsimple}
	\dba\bv +L^2 \del\bv = \bw,
\end{align}
 in the  weighted spaces $l^{2,\frac{p}{2}}(\BN;H^q(\OM)) $ for arbitrary positive integers  $p$ and $q$, where $L$ denotes the left translation operator $L\bu= L (u_0,u_{-1},u_{-2},...):=(u_{-1},u_{-2},...)$ and 
 $	\dba = \frac{1}{2}(\del_{x_1}+\i \del_{x_2}), \,  \del =\frac{1}{2}(\del_{x_1}-\i \del_{x_2}) $ are  the Cauchy-Riemann operators.
 
  Extension of the results in \cite[Theorem 4.2]{fujiwaraSadiqTamasan19} do not follow from differentiation of the equation, with the difficulty arising in the weighted estimates in $p$.    
  They are due to the fact that the sequence valued maps $\ds \OM \ni z \mapsto (v_{-1}(z),2^pv_{-2}(z), \cdots, n^p v_{-n}(z), \cdots )$ no longer solve a Beltrami equation of the type \eqref{BukhBeltsimple}. 
   However, a priori  estimates of traces of higher order gradients of solution of \eqref{BukhBeltsimple} are necessary. 
 To address this difficulty we first introduce a hierarchy of norms defined inductively  for $p\geq 0$ integer  as follows:
On the set of sequences with elements in $H^q(\OM)$, let $ S^{0}(\BN; H^q(\OM))$ be the space of square summable sequences $\bu$ endowed  with the norm
\begin{align}\label{newnorm}
	\llnorm{\bu}_{0,q}&:=\left(\sum_{j=0}^\infty \|u_j\|_{ H^q(\OM)}^2\right)^{\frac{1}{2}}, 
\end{align}
and for $p\geq 1$, let 
\begin{align*}
	S^{\frac{p}{2}}(\BN; H^q(\OM)) &: = \left \{\bu= \langle u_{0}, u_{-1},u_{-2},... \rangle: \llnorm{\bu}_{\frac{p}{2},q}  < \INF \right \},
\end{align*} where 
\begin{align}\label{newnorm1}
\llnorm{\bu}_{\frac{p}{2},q}&:= \left(\sum_{n=0}^{\INF} \llnorm{L^n \bu}_{\frac{p-1}{2},q}^2 \right)^{\frac{1}{2}}.
\end{align}

The following result shows that the norm in \eqref{weightedSpaces} is  equivalent to the one defined inductively above,
yielding  $S^{\frac{p}{2}}(\BN; H^q(\OM))= 	l^{2,\frac{p}{2}}(\BN;H^q(\OM))$. 

\begin{lemma}\label{lem:equivspaces}

Let $\lnorm{\cdot}^2_{\frac{p}{2},q} $ be the norm in \eqref{weightedSpaces} and $\llnorm{\cdot}^2_{\frac{p}{2},q} $ be the norm defined inductively in  \eqref{newnorm1}, for some integers $p,q\geq 0$.
	Then
	\begin{align}\label{newnorm2}
		\llnorm{\bu}_{\frac{p}{2},q}& 		\approx  \lnorm{\bu}_{\frac{p}{2},q}.
	\end{align}
\end{lemma}
\begin{proof}
	We first show  by induction in $p$ (and $q$ fixed) the equality 
		\begin{align}\label{newnorm3}
	\llnorm{\bu}^2_{\frac{p}{2},q} & =  \sum_{j=0}^{\INF} C^{j+p}_{p}\; \lnorm{u_j}_{ H^q(\OM)}^2,
    \end{align}where $C^{j+p}_p =\frac{(j+p)!}{j! p!}$.

	The case $p=0$ holds by definition \eqref{newnorm}.
	
	Assume next that the equality in \eqref{newnorm3} holds for some fixed $p$.

By definition \eqref{newnorm1},
	\begin{align}\label{newnorm1_hyp+11}
		\llnorm{\bu}_{\frac{p+1}{2},q}^2&= \sum_{n=0}^{\INF} \llnorm{L^n \bu}_{\frac{p}{2},q}^2 =  
		\sum_{m=0}^{\INF} \sum_{n=0}^{\INF}  C^{m+p}_{p} \; \lnorm{u_{m+n}}_{ H^q(\OM)}^2.
	\end{align}
	By changing the index $j=m+n$, for $m\geq 0$, ($j-n \geq 0,$ and $n \leq j$) we get
	\begin{align}\label{newnorm1_hyp+12}
		\sum_{m=0}^{\INF} \sum_{n=0}^{\INF} C^{m+p}_{p}\; \lnorm{u_{m+n}}_{ H^q(\OM)}^2= \sum_{j=0}^{\INF} \sum_{n=0}^{j} C^{j-n+p}_{p}\; \lnorm{u_{j}}_{ H^q(\OM)}^2 = \sum_{j=0}^{\INF}
		\lnorm{u_{j}}_{ H^q(\OM)}^2		
		\sum_{n=0}^{j} C^{j-n+p}_{p}.
	\end{align}
	
	By using Pascal's recurrence and the telescopic cancellations,
$$\ds \sum_{n=0}^{j} C^{j-n+p}_p 
	=  C^{j+p+1}_{p+1}.$$
	Thus, from \eqref{newnorm1_hyp+12}, 
	$\ds \llnorm{\bu}_{\frac{p+1}{2},q}^2= \sum_{j=0}^{\INF} C^{j+p+1}_{p+1} \lnorm{u_j}_{ H^q(\OM)}^2		.$

The equivalence of the norms in 	\eqref{newnorm2} follows from 
$\ds 	\frac{1}{p!}(1+j)^p\leq C^{j+p}_{p}  \leq (1+j)^p.$
\end{proof}



The following result establishes the base case in the  bootstrapping in the decay in $p$ of solutions of 
 \eqref{BukhBeltsimple}.

\begin{theorem}\label{Th_Est_delnu} 
	Let $\bw \in l^{2,\frac{p}{2}+1}(\BN;L^2(\OM))$  for some fixed integer $p \geq0$. 
	If $\bv \in l^{2,\frac{p+1}{2}}(\BN;H^1(\OM))$   solves  \eqref{BukhBeltsimple}, then 
	\begin{align}\label{estimate_newnorm}
		\lnorm{\bv}_{\frac{p}{2},1}^2 \lesssim \lnorm{\bw}_{\frac{p}{2}+1,0}^2+ \lnorm{\bv \lvert_{\Gam}}^2_{\frac{p+1}{2},\frac{1}{2}}.
	\end{align}

\end{theorem}

\begin{proof}
	We reason by induction in $p$.  
	The case $p=0$, 
	\begin{align}\label{estimateSimple}
		\lnorm{\bv}_{0,1}^2 \lesssim \lnorm{\bw}_{1,0}^2+ \lnorm{\bv \lvert_{\Gam}}^2_{\frac{1}{2},\frac{1}{2}},
	\end{align}   
	is established in \cite[Corollary 4.1]{fujiwaraSadiqTamasan19}. 
	
	Assume next that \eqref{estimate_newnorm} holds for $p$:
	\begin{align}\label{estimate_newnorm_p}
			\lnorm{\bv}_{\frac{p}{2},1}^2 \lesssim \lnorm{\bw}_{\frac{p}{2}+1,0}^2+ \lnorm{\bv \lvert_{\Gam}}^2_{\frac{p+1}{2},\frac{1}{2}}.
\end{align}   
	
	Next we bootstrap the decay in $p$ in  \eqref{estimate_newnorm_p} by $\frac{1}{2}.$
	Since $\bv$ solves \eqref{BukhBeltsimple}, 
	for each $n\geq 0$, the left shifted sequence $L^n \bv$ solves the shifted inhomogeneous Bukhgeim-Beltrami equation
	\begin{align}\label{BukhBelt_Ln}
		\dba L^n\bv +L^2 \del L^n\bv = L^n\bw.
	\end{align} 
		Thus, it satisfies the estimate \eqref{estimate_newnorm_p}
	with $\bv$ replaced by  $L^n \bv$, and $\bw$ replaced by $L^n\bw$. 	A summation over $n$ then yields 
	\begin{align*}
		\sum_{n=0}^{\INF}     	\lnorm{L^n\bv}_{\frac{p}{2},1}^2  \lesssim 
		\sum_{n=0}^{\INF}  \lnorm{L^n\bw}_{\frac{p}{2}+1,0}^2+ 	\sum_{n=0}^{\INF} \lnorm{L^n\bv \lvert_{\Gam}}^2_{\frac{p+1}{2},\frac{1}{2}},
	\end{align*} 
	which in view of the norm equivalence \eqref{newnorm2} rewrites  as
	\begin{align}\label{estimate_newnorm_phalf}
		\lnorm{\bv}_{\frac{p+1}{2},1}^2 \lesssim \lnorm{\bw}_{\frac{p+1}{2}+1,0}^2+ \lnorm{\bv \lvert_{\Gam}}^2_{\frac{p+2}{2},\frac{1}{2}}.
	\end{align} 
	By hypothesis, the right-hand-side is finite.			
	
\end{proof}


	For higher regularity estimates in the spatial variable, we need an estimate  on the traces of higher order derivatives on the boundary. Suffices to consider the equation \eqref{BukhBeltsimple} away from the origin 
	and work in polar coordinates.

	Let  ${\Omega}_\epsilon=\{ \epsilon < |z|<1 \}$. In $\Omega_\epsilon$ the  Cauchy-Riemann operators rewrite in terms of the angular derivative  $\del_{\eta}$ and the radial derivative $\del_r$ as
	\begin{align*}
		\del = \frac{e^{-\i \eta}}{2} \left(\del_r -\frac{\i}{r} \del_\eta \right)
		\quad \text{ and} \quad
		\ol{\del}  = \frac{e^{\i \eta}}{2} \left(\del_r +\frac{\i}{r} \del_\eta \right), 
	\end{align*} and the inhomogeneous Bukhgeim-Beltrami equation \eqref{BukhBeltsimple} becomes
	\begin{align}\label{BBeq_interms_delR_ETA}
		A\del_r \bv  =  -\frac{\i}{r} B\del_\eta \bv + 2 e^{\i \eta}\bw, \quad \epsilon <r <1,
	\end{align} where 
	\begin{align}\label{AB_defn}
		A:= e^{2\i \eta} +   L^2, 
		\quad \text{ and} \quad
		B := e^{2\i \eta} -  L^2.
	\end{align} 
	
	While it is easy to see that $\ds A,B: 	l^{2,p}(\BN;H^q(\OM)) \longrightarrow 	l^{2,p}(\BN;H^q(\OM))$ are bounded operators,
they are not invertible on  $l^{2,p}(\BN;H^q(\OM))$. The problem is that the unit circle lies in the spectrum of the left translation $L^2$. However, $A$ will be invertible on a proper subspace as follows. We start with a general result which may be of independent interest.
 \begin{lemma}\label{Aresult_AlexThesis}
		Let $p\geq 0$ be an integer, $L$ be the left shift operator, and $\lambda \in \BC$ with $|\lambda |\geq1$. Let  $\ba \in l^{2,p}$ and $\bc \in l^{2,p+1}$ be sequences satisfying $\ds 	(\lambda +L^2) \ba = \bc.$
		Then 
		\begin{align*}
			\lnorm{\ba}_{l^{2,p}} \leq  \frac{2^{p+1}}{2p+1} \lnorm{\bc}_{l^{2,p+1}}.
		\end{align*}
	\end{lemma}	
\begin{proof}

Since $a_j =\sum_{k=0}^\infty c_{j+2k}\lambda^{-k-1}$,
suffices to show the stronger estimate
\begin{align}\label{discreteineq}
	\left\{\sum_{j=0}^\infty (1+j)^{2p}\left(\sum_{k=0}^\infty |c_{j+k}|\right)^2  \right\} ^{\frac{1}{2}}
	\leq  \frac{2^{p+1}}{(2p+1)} 
	\left\{ \sum_{j=0}^\infty (1+j)^{2p+2}|c_j|^2  \right\} ^{\frac{1}{2}}.
\end{align}
We prove first the continuous version:
\begin{align*}
	&\left\{\int_1^\infty (1+s)^{2p}\left(\int_s^\infty f(t)dt\right)^2 ds\right\}^{\frac{1}{2}}=
	\left\{\int_1^\infty\left(\int_1^\infty s(1+s)^{p}f(s\zeta)d\zeta\right)^2 ds\right\}^{\frac{1}{2}} \\
	& \qquad \leq  2^p	\left\{\int_1^\infty\left(\int_1^\infty s^{p+1}f(s\zeta)d\zeta\right)^2 ds\right\}^{\frac{1}{2}}
   \leq  2^p \int_1^\infty\left\{\int_1^\infty s^{2p+2} f^2(s\zeta)ds\right\}^{\frac{1}{2}}d\zeta
   \\
   	&\qquad   = 2^p
	\int_1^\infty\left\{\int_\zeta^\infty \frac{t^{2p+2}}{\zeta^{2p+3}} f^2(t)dt\right\}^{\frac{1}{2}}d\zeta 
	\leq 2^p
	\left\{\int_1^\infty\zeta^{-p-\frac{3}{2}}d\zeta
	\right\}
	\left\{\int_1^\infty t^{2p+2}f^2(t)dt\right\}^{\frac{1}{2}}
	\\
	&\qquad 
	\leq \frac{2^{p+1}}{2p+1}
	\left\{\int_1^\infty (1+t)^{2p+2}f^2(t)dt\right\}^\frac{1}{2},
\end{align*}
where the second inequality uses Minkowski's.

 By setting $f(t)=|c_{[t]}|$, for$[t]$ the largest integer smaller then $t$, we obtain \eqref{discreteineq}.
\end{proof}	


\begin{prop}\label{prop_Ainv}
		Let  $p\geq 0$ be an integer, and $A,B$ as in \eqref{AB_defn}. Then for any integer $ q \geq 0$, 
			\begin{itemize}\label{Ainv_qmap}
\item[(i)]
	$A^{-1}: 	l^{2,p+q+1}(\BN;H^{q}(\OM_{\epsilon})) \longrightarrow 	l^{2,p}(\BN;H^q(\OM_\epsilon))$,
	
	\item[(ii)]
	$\ds A^{-1} \frac{1}{r} B  \del_{\eta}: 	l^{2,p+q+1}(\BN;H^{q+1}(\OM_\epsilon)) \longrightarrow 	l^{2,p}(\BN;H^q(\OM_\epsilon))$.
	\end{itemize} 
	\end{prop}	
\begin{proof}[Proof of (i)]
	The case $q=0$ in \eqref{Ainv_qmap}
  follows directly from Lemma \ref{Aresult_AlexThesis} and an integration over $\OM_\epsilon$.
	
	Let $q\geq 1$ be arbitrarily fixed.
	
	For any $\bc \in 	l^{2,p+q+1}(\BN;H^q(\OM_\epsilon))$, we have
		\begin{align}\label{Ainvc_exp}
		 \left( \del^{\alpha}_r A^{-1} \bc(r, \eta) \right)_j= \sum_{k=0}^\infty \del^{\alpha}_r c_{j+2k}(r, \eta) e^{-\i(2k+1)\eta}, \quad 0 \leq \alpha \leq q.
	\end{align}
	An application of Lemma \ref{Aresult_AlexThesis} and an integration over $\OM_\epsilon$ yields 
		\begin{align}\label{derAinv_map}
	&	\del_{r}^\alpha A^{-1}: 	l^{2,p+1}(\BN;H^{q}(\OM_\epsilon)) \longrightarrow 	l^{2,p}(\BN;H^{q-\alpha}(\OM_\epsilon)). 
\end{align} 
	
	For $ 0 \leq \beta \leq q$, by Leibniz formula, 
	$$\ds	\del_{\eta}^\beta \left( A^{-1} \bc \right)_j= \sum_{s=0}^{	\beta} (-\i)^s 
	C^{\beta}_s \sum_{k=0}^\infty \left( \del_{\eta}^{\beta-s }c_{j+2k} \right) (2k+1)^s e^{-\i(2k+1)\eta},$$
where $C^{\beta}_s =\frac{\beta!}{(\beta-s)! s!}$.

  An integration over $\OM_{\epsilon}$ in 
    \begin{align*}
 \sum_{j=0}^\infty (1+j)^{2p}	\lvert \del_{\eta}^\beta \left( A^{-1} \bc \right)_j  \rvert^2 
    &\lesssim
    \sum_{s=0}^{\beta} \sum_{j=0}^\infty (1+j)^{2p} \left( \sum_{k=0}^\infty  (k+1)^s \lvert \del_{\eta}^{\beta-s }c_{j+2k} \rvert \right)^2\\
   &\lesssim
  \sum_{s=0}^{\beta} \left [\sum_{j=0}^\infty (1+j)^{2p} \left( \sum_{k=0}^\infty (1+j+k)^s \left| \del_{\eta}^{\beta-s }c_{j+k} \right|   \right)^2 \right]\\
  &\lesssim
  \sum_{s=0}^{\beta} \sum_{j=0}^\infty (1+j)^{2p+2s+2} \left| \del_{\eta}^{\beta-s }c_{j} \right|^2
    \end{align*}
(where the last inequality uses Lemma \ref{Aresult_AlexThesis})
 yields \begin{align*}
	\lnorm{ \del_\eta^\beta  A^{-1} \bc}^2_{p,0} 
	\lesssim
	\sum_{s=0}^{\beta}  \lnorm{ \del_\eta^{\beta-s}  \bc}^2_{p+s+1,q-\beta}.
\end{align*}
Thus,   
		\begin{align}\label{etaqAinv_map}
	&	\del_{\eta}^\beta A^{-1}: 	l^{2,p+q+1}(\BN;H^{q}(\OM_\epsilon)) \longrightarrow 	l^{2,p}(\BN;H^{q-\beta}(\OM_\epsilon)).
\end{align} 

The mixed derivatives follow from  \eqref{derAinv_map} and \eqref{etaqAinv_map} as well, 
	\begin{align}\label{etaqAinv_map1}
	& \del_{r}^{\alpha}	\del_{\eta}^{q-\alpha} A^{-1}: 	l^{2,p+q+1}(\BN;H^{q}(\OM_\epsilon)) \longrightarrow 	l^{2,p}(\BN;L^{2}(\OM_\epsilon)),
\end{align} 
 thus yielding \eqref{Ainv_qmap}.

\noindent {\it Proof of (ii)}: Since 
$ \del_{\eta} : 	l^{2,p}(\BN;H^{q+1}(\OM_\epsilon)) \longrightarrow 	l^{2,p}(\BN;H^q(\OM_\epsilon))$ and 
$ B : 	l^{2,p}(\BN;H^{q}(\OM_\epsilon)) \longrightarrow 	l^{2,p}(\BN;H^q(\OM_\epsilon))$ are bounded, by part (i), 
		\begin{align*}
&A^{-1} \frac{1}{r} B \del_{\eta}: 	l^{2,p+q+1}(\BN;H^{q+1}(\OM_\epsilon)) \longrightarrow 	l^{2,p}(\BN;H^q(\OM_\epsilon)).
\end{align*} 
\end{proof}

Estimates of higher regularity in the spatial variable use first an estimate on the traces of higher order derivatives at the boundary.

\begin{cor}\label{lem_delnu}
	Let $\bw \in l^{2,\frac{p}{2}+1}(\BN;H^q(\OM_\epsilon))$, for some $p \geq 0 $, $q \geq1$ integers.
	If $\bv \in l^{2,\frac{p}{2}+1}(\BN;H^{q+1}(\OM_\epsilon))$ solves  \eqref{BukhBeltsimple}, then 
	\begin{align}\label{est_delnu}
		&\lnorm{ \nabla^q \bv \lvert_{\Gam}}^2_{\frac{p}{2}, \frac{1}{2} } 
		\lesssim
		\lnorm{  \bv \lvert_{\Gam} }^2_{\frac{p}{2}+q+1, q+\frac{1}{2}}
		+ 	\lnorm{ \bw \lvert_{\Gam} }^2_{\frac{p}{2}+q+1, q-\frac{1}{2}},
	\end{align}
	where $\nabla$ stands for either $ \ol{\del} $ or $\del$.
\end{cor}

\begin{proof} Recall 
\eqref{BBeq_interms_delR_ETA}:
 $$\ds	A \del_r \bv  = - \frac{\i}{r} B\del_\eta \bv + 2e^{\i \eta}\bw,$$ with $A,B$ as in \eqref{AB_defn}. 

By Proposition \ref{prop_Ainv} ( applying $A^{-1}$ and $ 	\del_{\eta}^{\beta } \del_{r}^{\alpha}$  with $\alpha+\beta +1=q$),  
		\begin{align*}
	 	\del_{\eta}^{\beta } \del_{r}^{\alpha} \del_r	\bv  = -\i \del_{\eta}^{\beta } \del_{r}^{\alpha}	 A^{-1}\frac{1}{r} B\del_\eta \bv + 2 \del_{\eta}^{\beta } \del_{r}^{\alpha} A^{-1} e^{\i \eta}\bw
	\end{align*}  holds in $l^{2,p}(\BN;H^{1}(\OM_\epsilon))$. By taking the trace on $\Gamma$  yields the estimate
	\begin{align}\label{normest_delnu}
		&\lnorm{ \nabla^q \bv \lvert_{\Gam}}^2_{p,\frac{1}{2}}
		\lesssim
		\lnorm{  \bv \lvert_{\Gam} }^2_{p+q+1, q+\frac{1}{2}}+
		\lnorm{ \bw \lvert_{\Gam} }^2_{p+q+1, q-\frac{1}{2}}.
	\end{align}
	
	The estimate \eqref{normest_delnu} is extended in the $p$ index from non-negative integers to multiples of $\frac{1}{2}$ as follows. 
	 
		Since $L^n \bv$ solves the shifted inhomogeneous Bukhgeim-Beltrami equation \eqref{BukhBelt_Ln}, we can apply the  estimate \eqref{normest_delnu} with $\bv$ replaced by $L^n \bv$ and 
	$\bw$ replaced by $L^n \bw$.
	A summation over $n$ yields 
	\begin{align}\label{Ln_Delrq_onehalfnorm}
		\sum_{n=0}^{\INF}     	
		\lnorm{L^n  \nabla^q \bv \lvert_{\Gam}}^2_{p,\frac{1}{2}}
		 \lesssim \sum_{n=0}^{\INF}    
		\lnorm{  L^n \bv \lvert_{\Gam}}^2_{p+q+1,q+\frac{1}{2}} +\sum_{n=0}^{\INF}    \lnorm{ L^n  \bw \lvert_{\Gam} }^2_{p+q+1,q-\frac{1}{2}}.
	\end{align} 
	
	Each of the sums in the inequality \eqref{Ln_Delrq_onehalfnorm} above are nothing else but the inductively defined norms  \eqref{newnorm1}. 	In conjunction with the  norm equivalence \eqref{newnorm2}, the inequality \eqref{Ln_Delrq_onehalfnorm} rewrites:
	\begin{align*}
		\lnorm{ \nabla^q \bv \lvert_{\Gam} }^2_{p+\frac{1}{2},\frac{1}{2}}  \lesssim
		\lnorm{ \bv \lvert_{\Gam} }^2_{p+q+\frac{3}{2},q+\frac{1}{2}} +\lnorm{ \bw \lvert_{\Gam} }^2_{p+q+\frac{3}{2},q-\frac{1}{2}}. 
	\end{align*}
Since $p$ was an arbitrary integer, this is the identity \eqref{est_delnu}.
\end{proof}

   We are now able to prove a general interior estimate for solution of \eqref{BukhBeltsimple}  in terms of their traces on the boundary. 
	They are crucial in the bootstrapping argument.

\begin{theorem}\label{lem_Est_delnu}  
		Let $ \bw \in l^{2,\frac{p}{2}+q+\frac{3}{2}}(\BN;H^q(\OM)) $  
for some $p \geq 0 $ and $q \geq1$ integers.
If $ \bv \in l^{2,\frac{p}{2}+q+\frac{3}{2}}(\BN;H^{q+1}(\OM)) $ solves  \eqref{BukhBeltsimple}, then
\begin{align} \label{estimate_bvnorm_zerotwo_2}
	\lnorm{\bv}_{\frac{p}{2},q+1}^2 &\lesssim \lnorm{\bw}_{\frac{p}{2}+1,q}^2+ \lnorm{\bv \lvert_{\Gam}}^2_{\frac{p}{2}+q+\frac{3}{2},q+\frac{1}{2}} +  \lnorm{ \bw \lvert_{\Gam}}^2_{\frac{p}{2}+q+\frac{3}{2},q-\frac{1}{2}}.
\end{align}   
\end{theorem}

\begin{proof}

	Since $\nabla^q \bv$ solves 
	\begin{align*}
		\dba (\nabla^q\bv) +L^2 \del (\nabla^q \bv) = \nabla^q \bw,   
	\end{align*}  we apply Theorem \ref{Th_Est_delnu}
  to $\nabla^q\bv$:
	\begin{align*}\label{estimate_bvnorm_zerotwo}
		\lnorm{\bv}_{\frac{p}{2},q+1}^2 &\lesssim \lnorm{\bw}_{\frac{p}{2}+1,q}^2+ 
		+ \lnorm{\nabla^q\bv \lvert_{\Gam}}^2_{\frac{p}{2}+\frac{1}{2},\frac{1}{2}}.
	\end{align*}   
	A further application of  Corollary \ref{lem_delnu}  to the last term establishes \eqref{estimate_bvnorm_zerotwo_2}.

\end{proof}


\section{Proof of Theorem \ref{Th_main} in the non-attenuated ($a=0$) case }\label{sec:Thproof_NART}

It is easy to see (e.g., in \cite[Lemma A.1]{sadiqTamasan23}) that 
\begin{equation}\label{eq:innerprod_fThetaM}
	\ds \langle \bbf, \btheta^m \rangle   =
				\begin{cases} 
		\ds f_0 +\sum_{k=1}^{\frac{m}{2}} \left( f_{2k} e^{-\i (2k)\tta} +  f_{-2k} e^{\i (2k)\tta} \right), & \text{if } m = \text{even}, \\
		\ds \sum_{k=0}^{\frac{m-1}{2}} \left( f_{2k+1} e^{-\i (2k+1)\tta} + f_{-(2k+1)} e^{\i (2k+1)\tta}\right), & \text{if } m =\text{odd},
	\end{cases}
\end{equation} 
for some functions
$\{f_{k}:\; -m\leq k\leq m\}$ 
in an explicit one-to-one correspondence (linear combination) with $\{ f_{ \underbrace{1\cdots1}_{m-k} \underbrace{2\cdots2}_{k}}:\; 0\leq k\leq m\}\ds$. For symmetric tensors, the latter coincide with $f_{i_1 \cdots i_m} $ for all multi-indexes $(i_1, \cdots, i_m)\in \{1,2\}^m$ in which  $2$ occurs exactly $k$ times.

We present the proof  for the even order tensors, while for the odd order case we exhibit the nominal changes only.

Let  $m$ be an even integer and the attenuation $a\equiv 0$. 

In our inverse problem, the solution $\ds  v^k, \; 0 \leq k \leq m$ of the boundary value problem \eqref{bvp_UK_transport} with $a \equiv 0$ is unknown in $\OM$, since $m$-tensors  $\bbf$ is unknown. However, their traces  	
\begin{align}\label{trace:gk}
	g^k=
	\begin{cases} 
		v^k \lvert_{\Gam_{+}} &\mbox{ on } \Gam_{+},\\
		0 &\mbox{ on } \Gam_{-},
	\end{cases}
\end{align}
are known on $\Gam \times \sph$ 
from the momenta data $\langle I^0\bbf, I^1\bbf,  \cdots , I^m\bbf \rangle$ via \eqref{eq:gk-Ik} for $a \equiv 0$ therein:
\begin{align}\label{data_equiv}
		 \langle g^0, g^1,   \cdots ,  g^m \rangle \longleftrightarrow \langle I^0\bbf, I^1\bbf,  \cdots , I^m\bbf \rangle.
\end{align}


While unknown, by Proposition \ref{prop_transEq} the a priori smoothness assumption on  $\bbf$ yield 
$\ds v^k \in H^{m+\frac{3}{2}}(\Omega\times\sph)$. Thus $ \ds g^k, I^k\bbf \in H^{m+\frac{3}{2}}(\sph; H^{m+\frac{1}{2}}(\Gam) )$ for $0 \leq k \leq m$.

For the inverse source problem we work with the sequence of the Fourier coefficients of the $k$-level flux $v^k(z, \cdot)$, in the angular variable:
\begin{align}\label{eq:vkn}
v^k_n(z)=\frac{1}{2\pi} \int_{-\pi}^{\pi} v^k(z,\btheta ) e^{-\i n\theta}d\theta, \quad n \in \BZ, \, 0 \leq k \leq m.
\end{align}
The upper index $k$ denotes the level of the flux, while the lower index $n$ is the Fourier coefficient in the angular variable.

For $\theta=\arg{\btheta}\in (-\pi,\pi]$, the advection operator in polar coordinates becomes $\btheta \cdot \nabla = e^{-\i \tta}\dba + e^{\i \tta}\del$.
By identifying the Fourier coefficients in \eqref{bvp_UK_transport} and  by using  \eqref{eq:innerprod_fThetaM}, the  modes $v^k_n$'s solve
	\begin{subequations}\label{Beltrami_modes_vk_sys}
\begin{align} \label{NART_mEven_fmEq}
	&\ol{\del} v^0_{-(2n-1)}(z) + \del v^0_{-(2n+1)}(z)  = f_{2n}(z), && 0 \leq n \leq m/2, \\ 
	\label{NARTmEvenAnalyticEq_Odd}
	&\ol{\del} v^0_{-(2n-1)}(z) + \del v^0_{-(2n+1)}(z)  = 0, 		&& n \geq m/2+1,\\ 
		\label{NARTmEvenAnalyticEq_Even}
	&\ol{\del} v^0_{-2n}(z) + \del v^0_{-(2n+2)}(z) = 0, &&  n \geq 0, \\ 
	\label{Beltrami_vkEq}
	&\ol{\del} v^k_{-n}(z) + \del v^k_{-n-2}(z) = v^{k-1}_{-n-1}(z),
	&& n \in \BZ, \: 1\leq k \leq m,
\end{align}
and 
\begin{align}\label{trace:gkn1}
	v^k_{-n} \lvert_{\Gam }= g^k_{-n}, \qquad n \in \BZ, \: 0\leq k \leq m.
\end{align}
	\end{subequations}

The existence of the solution to the boundary value problem \eqref{Beltrami_modes_vk_sys}  is postulated by the forward problem.

Since $\bbf$ is real valued, the solution $\ds v^k $ of  \eqref{bvp_UK_transport}  is also  real valued, and its Fourier modes in the angular variable occur in conjugates:
\begin{align}\label{reality_un_complexconjugate}
	v^k_{-n} = \ol{v^k_{n}}, \quad \text{ for }  \; n \geq 0, \,0\leq k \leq m. 
\end{align}	Thus, it suffices to consider the non-positive Fourier modes of $v^k(z,\cdot)$.

For $0\leq k\leq m$, let $\bv^k$ be the sequence valued map of the non-positive Fourier coefficients of the solution  $v^k$ and $\bg^k$ be its corresponding trace on the boundary:
\begin{align}\label{eq:boldvk}
	\bv^k(z)&:= \langle v^k_{0}(z), v^k_{-1}(z), v^k_{-2}(z),  \cdots \rangle,\quad z \in \OM,\\
	\label{seq:gk}
	\bg^k &= \langle g^k_{0}, g^k_{-1}, g^k_{-2},  \cdots \rangle:= \bv^k \lvert_{\Gam}.
\end{align}

In the sequence valued map notation the boundary value problem \eqref{Beltrami_modes_vk_sys} becomes
	\begin{subequations}\label{Beltrami_vk_sys}
		\begin{align} \label{v0_1_f0Eq}
			\ol{\del} \overline{v^0_{-1}}+ \del v^0_{-1}  &= f_{0},\\    \label{Beltrami_v0_Feq}
			\dba\bv^0 +L^2 \del\bv^0 &= L\bF, \\ \label{Beltrami_vkseq_Eq}
			\dba\bv^k +L^2 \del\bv^k &= L \bv^{k-1}, \quad 1 \leq k \leq m, 
		\end{align}
		subject to 
		\begin{align}\label{boldgk_vk}
			\bv^{k} \lvert_{\Gam }&=\bg^{k}, \quad \text{for } \, 0 \leq k \leq m,
		\end{align}	
	\end{subequations}
		where 
	\begin{align}\label{eq:bF}
		\bF:= \langle f_0, 0 , f_{2}, 0, f_{4}, 0, \cdots, f_{m-2}, 0, f_m, 0, 0, \cdots \rangle
	\end{align}  is the sequence valued map  build on the Fourier modes $\{f_{2k} : 0 \leq k \leq \frac{m}{2}\} $ in  \eqref{eq:innerprod_fThetaM} for $m = $ even.

Note that the sequences $\ds  \bv^k, \; 0 \leq k \leq m$ in  \eqref{Beltrami_vk_sys}  are unknown, since $\bF$ is unknown. However,  their traces on $\Gam$ are known  from \eqref{data_equiv}.

   			It is crucial to note that $L^{m+1} \bF = \bzero= \langle 0,0,...\rangle$, so that for $\ds 0 \leq k \leq m$,  
   $L^{m-k}\bv^k$ solves
  the  boundary value problem
   \begin{subequations}\label{Shift_Beltrami_vk_sys}
   	\begin{align} \label{Shift_Beltrami_v0_Feq}
   		\dba[L^{m}\bv^0] +L^2 \del[L^{m}\bv^0] &= \bzero, \\ \label{Shift_Beltrami_vkseq_Eq}
   		\dba[L^{m-k}\bv^k] +L^2 \del[L^{m-k}\bv^k] &= L^{m-k+1}\bv^{k-1}, \quad 1 \leq k \leq m, 
   	\end{align}
   	subject to 
   	\begin{align}\label{boldLmgk_vk}
   		L^{m-k}\bv^{k} \lvert_{\Gam }&=L^{m-k}\bg^{k}, 		
   		\quad 0 \leq k \leq m,
   	\end{align}	
   \end{subequations}
  which does not involve the source $\bF$ (encoding the tensor)!

	The following result identifies shifts of the solutions of the boundary value problem \eqref{Beltrami_vk_sys}  which are determined solely from the boundary data $ \langle \bg^0, \bg^1,   \cdots , \bg^m\rangle $ as in \eqref{seq:gk} and not by the source $\bF$.
	
	\begin{prop}\label{prop_estimate_delvk}
		Let $ \langle \bg^0, \bg^1,   \cdots , \bg^m\rangle $ be the data as in \eqref{seq:gk} obtained for some unknown even order $m$-tensor in $ H^{m+\frac{3}{2}}_0(\mathbf{S}^m; \OM)$. Then $\bg^k \in l^{2,m+\frac{3}{2}}(\BN;H^{m+\frac{1}{2}}(\Gam))$, and the  unique solution $\ds \bv^k$  of the boundary value problem \eqref{Beltrami_vk_sys} 
		satisfies
		\begin{equation}\label{eq:Lm-k_vk1}
			\begin{aligned}
				L^{m-k}	\bv^k (z)&=  \sum_{j=0}^{k} \BT^j L^{m-k+j}  [\B \bg^{k-j}](z), \quad z \in \OM, \; 0 \leq k \leq m,
			\end{aligned}
		\end{equation}	
		where $\B$ is the Bukhgeim-Cauchy operator  in \eqref{BukhgeimCauchyFormula}, 
		and $\BT$ is the operator in \eqref{T_Formula}. 
		Moreover, 
		\begin{equation}\label{estimate_delvk}
			\begin{aligned}
				\lnorm{ L^{m-k}   \bv^k}_{m-k,k+1}^2 	&\lesssim 
					\sum_{j=0}^{k}	\lnorm{  L^{m-j}  \bg^{j}}^2_{m +\frac{3}{2},j+\frac{1}{2}}, \quad 0 \leq k \leq m.
			\end{aligned} 
		\end{equation}
	\end{prop}
	\begin{proof}
		
			The existence of the solution to the boundary value problem \eqref{Beltrami_vk_sys} is postulated by the forward problem \eqref{bvp_UK_transport} (with $a\equiv 0$). 
			Moreover, by Proposition \ref{prop_transEq}, 
		the unknown $k$-level flux  $v^k 	\in H^{m+\frac{3}{2}}(\sph; H^{m+1}(\OM))$, yielding 
		$\bv^k \in l^{2,m+\frac{3}{2}}(\BN;H^{m+1}(\OM))$ and $\bg^k \in l^{2,m+\frac{3}{2}}(\BN;H^{m+\frac{1}{2}}(\Gam))$.

Recall that 
		$L^{m-k}\bv^k$ solve
		the system \eqref{Shift_Beltrami_vk_sys}.
		
		First we prove the formula \eqref{eq:Lm-k_vk1} with estimate \eqref{estimate_delvk} by induction in $k$,  for  $0 \leq k \leq m$.
		
       Case $k=0$: Since $L^{m} \bv^0 $ is $L^2$-analytic, the Bukhgeim-Cauchy Integral formula \eqref{BukhgeimCauchyFormula} determines  the sequence $L^{m} \bv^0 $ 
		inside $\OM$ from its boundary values via
		\begin{align}\label{construction_Lmv0}
			L^{m}\bv^0(z) &	:= \B L^{m}\bg^0  (z), \quad z\in \OM.
		\end{align}
		Applying Theorem \ref{Th_Est_delnu}  to the boundary value problem \eqref{Shift_Beltrami_v0_Feq} and \eqref{boldLmgk_vk} yields 
		\begin{align}\label{estimate_delv0}
			\lnorm{ L^{m}  \bv^0}_{m,1}^2 	\lesssim \lnorm{ L^{m} \bg^0}^2_{m+\frac{1}{2},\frac{1}{2}},
		\end{align}
		thus showing the $k=0$ case.

	Next, we assume \eqref{eq:Lm-k_vk1} and \eqref{estimate_delvk} holds for $k$:
		\begin{equation}\label{eq:Lm-k_vk_1}
			\begin{aligned}
				L^{m-k}	\bv^k (z)&=  \sum_{j=0}^{k} \BT^j L^{m-k+j}  [\B \bg^{k-j}](z), \quad z \in \OM, 
			\end{aligned}
		\end{equation}
		satisfies 
		\begin{equation}\label{estimate_delvk_1}
			\begin{aligned}
				\lnorm{ L^{m-k}   \bv^k}_{m-k,k+1}^2 	&\lesssim 
				\sum_{j=0}^{k}	\lnorm{  L^{m-j}  \bg^{j}}^2_{m+\frac{3}{2},j+\frac{1}{2}}, 
			\end{aligned}
		\end{equation} and prove it  for $k+1$.

		Starting from  equation \eqref{Shift_Beltrami_vkseq_Eq} and \eqref{boldLmgk_vk} for $k+1$, $L^{m-(k+1)}\bv^{k+1}$ solves 
		\begin{subequations}\label{Beltrami_Lmvk_sys}
			\begin{align} \label{Beltrami_vkEq_2}
				\dba [L^{m-(k+1)}\bv^{k+1}] +L^2 \del [L^{m-(k+1)}\bv^{k+1}] = L^{m-k} \bv^{k},
			\end{align}
			subject to 
			\begin{align}\label{bold_Lmgk_Lmvk}
				L^{m-(k+1)} \bv^{k+1} \lvert_{\Gam }&=L^{m-(k+1)} \bg^{k+1}.
			\end{align}	
		\end{subequations}

		Applying  Proposition \ref{prop_Bukhgeimpompeiu}
		to \eqref{Beltrami_Lmvk_sys}, 
		 the solution $L^{m-(k+1)}\bv^{k+1}$   is given by \eqref{BP_Formula_vk}:
		\begin{align*} 
			L^{m-(k+1)}	\bv^{k+1} 
			&=     	\B L^{m-(k+1)} \bg^{k+1}+	
			\BT (L^{m-k}\bv^{k}).
		\end{align*} 
	
	     	Following directly from their definitions, the operators $L$ and $\B$ commute, and the operators $L$ and $\BT$ commute. 
		Using these commutating properties and the induction hypothesis \eqref{eq:Lm-k_vk_1} yields 
		\begin{align*} 
			L^{m-(k+1)}	&\bv^{k+1} 
			=     	L^{m-(k+1)} 	\left[\B  \bg^{k+1}\right]+	
			\BT (L^{m-k}\bv^{k})\\
			&=     	L^{m-(k+1)} 	\left[\B  \bg^{k+1}\right] +	
			\sum_{j=0}^{k} \BT^{j+1} L^{m-k+j}  [\B   \bg^{k-j}] 
			=     	 \sum_{j=0}^{k+1} \BT^{j} L^{m-(k+1)+j}  [\B   \bg^{k+1-j}] .
		\end{align*} 
		This finished the inductive step of \eqref{eq:Lm-k_vk_1}.
		
		Next, we show the induction step for the estimate \eqref{estimate_delvk_1} for the solution $L^{m-(k+1)}\bv^{k+1}$ of \eqref{Beltrami_Lmvk_sys}.

		By applying Theorem \ref{lem_Est_delnu}    to \eqref{Beltrami_Lmvk_sys} 		yields 
	\begin{align*}
		\lnorm{L^{m-k-1}  \bv^{k+1}}_{m-k-1,k+2}^2 		
		&\lesssim \lnorm{ L^{m-k}   \bv^k}_{m-k,k+1}^2 +  \lnorm{L^{m-k}   \bv^{k} \lvert_{\Gam} }^2_{m+\frac{3}{2},k+\frac{1}{2}} +  \lnorm{L^{m-k-1}   \bv^{k+1} \lvert_{\Gam} }^2_{m+\frac{3}{2},k+\frac{3}{2}} \\
		&\lesssim   \sum_{j=0}^{k}	\lnorm{  L^{m-j}  \bv^{j}\lvert_{\Gam} }_{m+\frac{3}{2},j+\frac{1}{2}}^2 +\lnorm{L^{m-k-1}   \bv^{k+1} \lvert_{\Gam} }^2_{m+\frac{3}{2},k+\frac{3}{2}}  \\
 &\lesssim 
\sum_{j=0}^{k+1}	\lnorm{  L^{m-j}  \bg^j }_{m+\frac{3}{2},j+\frac{1}{2}}^2,
	\end{align*}
	where the second inequality uses  the induction hypothesis \eqref{estimate_delvk_1}, while the last one is regrouping.
\end{proof}

\subsection{The reconstruction method} \label{sec:reconstruct}

We reconstruct the $m$-tensors $\bbf$ by first recovering $\bv^k, 0 \leq k \leq m$ in two steps, see Figure \ref{fig:Flow_Eventensor}.

\begin{itemize}
	\item \underline{Step I (Sweep down):} 
	
	Level by level, starting from $k=0$ to $k=m$, we recover $\ds L^{m-k} \bv^k$ by solving the boundary value problem
		\begin{align*} 
			\dba \left( L^{m-k} \bv^k \right) +L^2 \del \left( L^{m-k} \bv^k \right)  &=  L^{m+1-k}  \bv^{k-1}, \quad 1 \leq k \leq m,
		\end{align*}
		subject to 
		\begin{align*}
			L^{m-k}  \bv^{k} \lvert_{\Gam }= L^{m-k} \bg^{k}, \quad \text{for } \, 1\leq k \leq m.
		\end{align*}	
	
	Proposition \ref{prop_estimate_delvk} ensures that the unique solution $\ds L^{m-k} \bv^k$ given by \eqref{eq:Lm-k_vk1}  satisfies estimate
	\eqref{estimate_delvk}. 
		In particular,  when $k=m$,  the entire sequence $ \bv^m$ (not just some translation of it) is recovered from  the data $	\bg^j$, $0 \leq j \leq m$,
		\begin{equation}\label{eq:vm-gk1}
			\begin{aligned}
				\bv^m(z) &=  \sum_{j=0}^{m} [\BT L]^j  \left(\B  \bg^{m-j} \right) (z), \quad z \in \OM,
			\end{aligned}
		\end{equation}	
		with the estimate 
		\begin{align}\label{estimate_delvm}
			\lnorm{  \bv^m}_{0,m+1}^2 	&\lesssim 
			\sum_{j=0}^{m}	\lnorm{ L^{m-j}  \bg^{j}}^2_{m+\frac{3}{2},j+\frac{1}{2}}; 
		\end{align}
		where $\B$ and $\BT$ are the operators  in \eqref{BukhgeimCauchyFormula}, respectively in \eqref{T_Formula}. 
		
		\begin{figure}[ht!]
			\centering
			\includegraphics[scale=0.51]{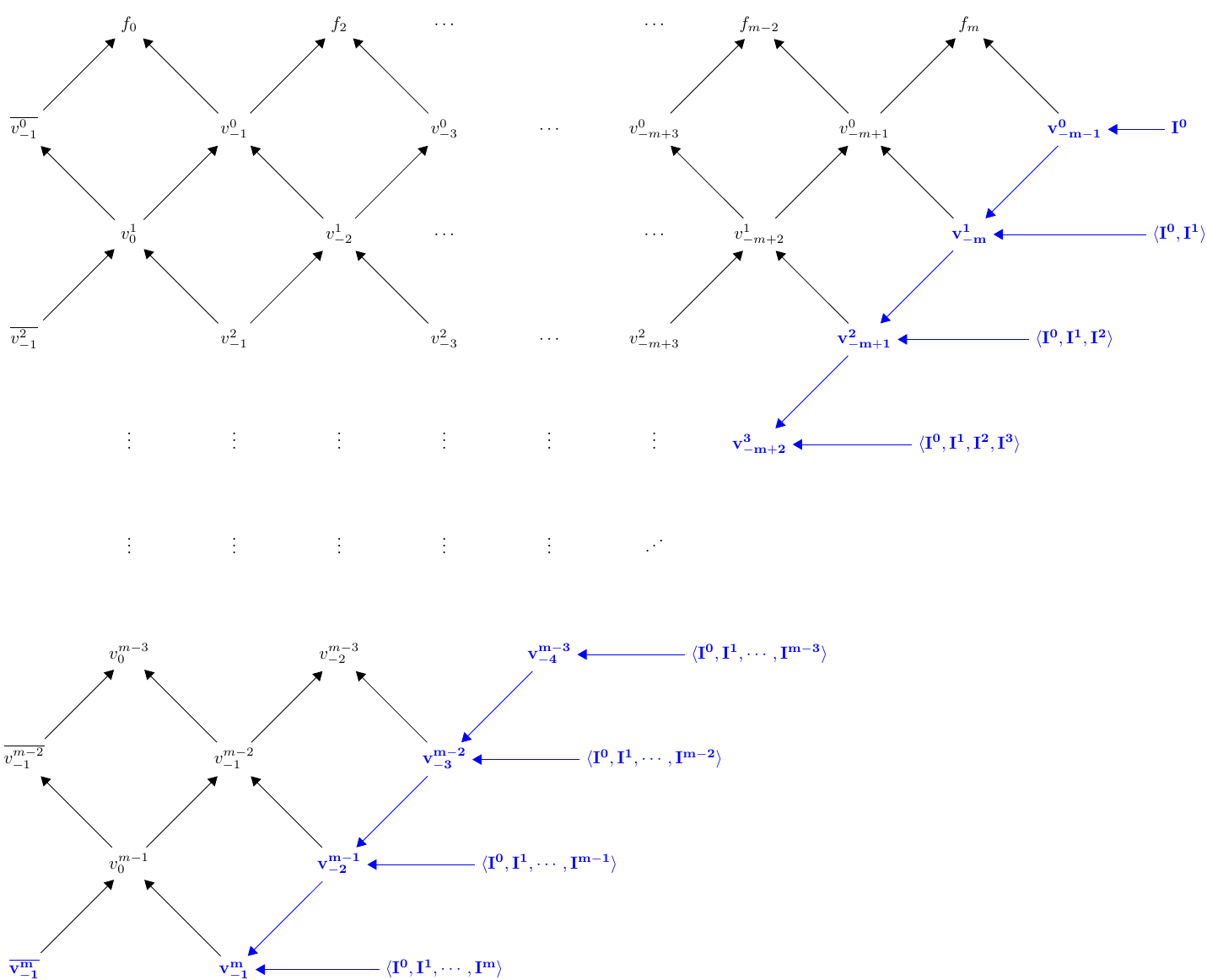}
			\caption{{Flow of the reconstruction of even order tensors. In the sweep down, all the modes colored in blue are determined layer by layer from the momenta ray data and the previous layer. 
					In the sweep up, also layer by layer starting from the bottom, the remaining coefficients are recovered. The arrows indicate which modes determine what.	
			}}
			\label{fig:Flow_Eventensor}
		\end{figure}


		\item \underline{Step II (Sweep up):} 
		
		Level by level, starting from  $k=m$ to $k=1$,  
  use \eqref{Beltrami_vkseq_Eq} in its component-wise form \eqref{Beltrami_vkEq},  
		\begin{align} \label{InhomBeltrami_vkEq1}
			v^{k-1}_{-n-1} &:=  \ol{\del} v^k_{-n} + \del v^k_{-n-2}, \quad  n \geq 0,
		\end{align} 
		to recover $L \bv^{k-1}$ from  the knowledge of $\bv^k$.
		
		Moreover, repeated differentiation of \eqref{InhomBeltrami_vkEq1} yields 
		\begin{align*} 
			\nabla^{q} \left( L \bv^{k-1} \right) = 	\dba [\nabla^{q}\bv^{k}] +L^2 \del [\nabla^{q}\bv^{k}], \quad  q \geq 1,
		\end{align*} 
		with the estimate  
		\begin{align}\label{estimate_vk_sweepup}
			\lnorm{ L \bv^{k-1}}_{0, q}^2 	\lesssim \lnorm{ \bv^{k}}_{0, q+1}^2.
		\end{align}

			We use \eqref{InhomBeltrami_vkEq1} recursively 
			to recover the entire sequences $\bv^{m-1},  \cdots, \bv^1, \bv^0$ with estimate 
			\begin{align}\label{estimate_v0_sweepup}
				\lnorm{ L \bv^{0}}_{0,1}^2 	&\lesssim \lnorm{ \bv^{m}}^2_{0,m+1}.
			\end{align} 
		
			\item \underline{The reconstruction of the $m$-tensors $\bbf$.}
			
			With $\bv^0$ known, we recover $\bF$ via \eqref{Beltrami_v0_Feq} and  \eqref{v0_1_f0Eq},
			\begin{equation}\label{eq:eGbF}
				\begin{aligned}
					f_{0}&:= 2 \re \left[ \del v^0_{-1} \right],  \quad \text{and} \quad 
					L\bF:= \dba\bv^0 +L^2 \del\bv^0.
				\end{aligned}
			\end{equation}
           Moreover, using in order  \eqref{eq:eGbF},    \eqref{estimate_v0_sweepup} and  \eqref{estimate_delvm}, we have the  estimate 
			\begin{align}\label{estimate_eGF_delu0}
				\lnorm{ \bF}_{0,0}^2  \lesssim 	 \lnorm{ \bv^{0}}_{0,1}^2 	&\lesssim \lnorm{ \bv^{m}}^2_{0,m+1} \lesssim 
				\sum_{j=0}^{m}	\lnorm{ L^{m-j}  \bg^{j}}^2_{m+\frac{3}{2},j+\frac{1}{2}}.
			\end{align} 
%

		\end{itemize}
		
		Finally, the components $f_{i_1 \cdots i_m} = f_{ \underbrace{1\cdots1}_{m-k} \underbrace{2\cdots2}_{k}}$
		are defined via the explicit one-to-one correspondence (linear combination) 
		with   $\{f_{2k}:\; -m/2\leq k\leq m/2\}$.
		
		Since $\ds \lnorm{ \bbf}_{L^2(\OM)}^2 \lesssim \lnorm{ \bF}_{0,0}^2  \lesssim \lnorm{ \bbf}_{L^2(\OM)}^2$, Theorem \ref{Th_main} is, thus,  proven for $m$ even and $a \equiv 0$.
		
		\qed

		The proof of Theorem \ref{Th_main} for odd $m$-order tensors follows similarly to the even case with some nominal changes:  As before, let $\bv^k $ 
	   be the sequence valued map of the Fourier coefficients of  the $k$-level flux solution  $v^k$ of the boundary value problem \eqref{bvp_UK_transport} (with $a \equiv 0$)  and 	$\ds \bg^k = \bv^k \lvert_{\Gam} $ 
		be its trace. 
		Then $\bv^k$ solves (contrast with the system \eqref{Beltrami_vk_sys})
		\begin{subequations}\label{Beltrami_vk_sys_NART_mOdd}
			\begin{align}  \label{Beltrami_v0_Feq_NART_mOdd}
				\dba\bv^0 +L^2 \del\bv^0 &= L\bF, \\ \label{Beltrami_vkseq_NART_mOdd}
				\dba\bv^k +L^2 \del\bv^k &= L \bv^{k-1}, \quad 1 \leq k \leq m, 
			\end{align}
			subject to 
			\begin{align}\label{boldgk_vk_NART_mOdd}
				\bv^{k} \lvert_{\Gam }&=\bg^{k}, \quad \text{for } \, 0 \leq k \leq m,
			\end{align}	
	where 
			\begin{align}\label{eq:bF_mODD}
		{\bF}:= \langle 0, f_1, 0 , f_{3}, 0, f_{4}, 0, \cdots, f_{m-2}, 0, f_m, 0, 0, \cdots \rangle 
	\end{align}  
 is  build on the Fourier modes $\{f_{2k+1} : 0 \leq k \leq \frac{m-1}{2}\} $ in  \eqref{eq:innerprod_fThetaM} for $m = $ odd.
%
	 	Note also the change in the definition of ${\bF}$ above from the one in the even tensor case in \eqref{eq:bF}. 		
				\end{subequations}
						
		 Proposition \ref{prop_estimate_delvk} holds verbatim for odd $ m$. 
		
		Following the two step reconstruction method in the even case,  $\bF$ and thus $\bbf$ are similarly recovered with the estimate \eqref{stability_f_mEven1}. The inversion of the momenta Doppler transform ($m=1$) is detailed in the numerical section \ref{sec:numerics}.
	


\section{Proof of Theorem \ref{Th_main} in the attenuated case 
}

As in \cite{sadiqTamasan01} we treat the attenuated case by the reduction to the non-attenuated case via the special integrating factor function   introduced in \cite{finch}:
\begin{align*}
	h(z,\btheta) := \int_{0}^{\INF} a(z+t\btheta)dt -\frac{1}{2} \left( I - \i H \right) Ra(z\cdot \btheta^{\perp}, \btheta^{\perp}),
\end{align*} where $H \psi(s,\btheta) = \ds \frac{1}{\pi} \int_{-\INF}^{\INF} \frac{\psi(t,\btheta)}{s-t}dt $ is the Hilbert transform 
taken in the linear variable, and $Ra(s, \btheta^{\perp}) =  \int_{-\INF}^{\INF} a\left( s \btheta^{\perp} +t \btheta \right)dt$ is the Radon transform of  $a$.

It is known that all the negative Fourier modes of $h$ vanish \cite{finch,nattererBook}, yielding 
\begin{align*}
	e^{- h(z,\btheta)} &:= \sum_{k=0}^{\INF} \alpha_{k}(z) e^{\i k\tta}, \quad e^{h(z,\btheta)} := \sum_{k=0}^{\INF} \beta_{k}(z) e^{\i k\tta}, \quad (z, \btheta) \in \ol\OM \times \sph. 
\end{align*}
In \cite{fujiwaraSadiqTamasan19}, 
 the convolution operators $e^{\pm G}$ are defined by
 \begin{equation} \label{eGop_leftshift}
\begin{aligned}
	e^{-G} \bu &:= \balpha \ast \bu \text{ and } e^{G} \bu := \bbeta \ast \bu,  \text{ where }  
\balpha := \langle \alpha_{0}, \alpha_{1},    ... \rangle, 
\; \bbeta := \langle \beta_{0}, \beta_{1},  ... \rangle. 
\end{aligned}
\end{equation}

The following result connecting the attenuated to the non-attenuated case is a slight generalization of \cite[Lemma  2.2]{fujiwaraSadiqTamasan19}.
\begin{lemma}\label{beltrami_reduction}
	Let $a\in C^{1,\mu}(\ol \OM)$, $\mu>1/2$.
	Then 
	$e^{\pm G}:l^{2,p}(\BN;H^q(\OM))\to l^{2,p}(\BN;H^q(\OM)) $
	are bounded. Moreover, 
	
		(i) if $\bu \in l^{2}(\BN;H^1(\OM))$  solves $\ds \dba \bu +L^2 \del \bu+ aL\bu = \bw$, then  $\ds \bv= e^{-G} \bu \in l^{2}(\BN;H^1(\OM))$  solves $\dba \bv + L^2\del \bv =  e^{-G} \bw$;
	
	(ii) Conversely, if $\bv \in l^{2}(\BN;H^1(\OM))$  solves $\dba \bv + L^2\del \bv =e^{-G} \bw$, then $\ds \bu= e^{G} \bv \in l^{2}(\BN;H^1(\OM))$  solves $\ds \dba \bu +L^2 \del \bu+ aL\bu = \bw$.
\end{lemma}
The mapping properties of  $e^{\pm G}$ and the case $\bw=\bzero$ in (i) and (ii) are proven in \cite{fujiwaraSadiqTamasan19}. 
The case $\bw\neq \bzero$ follows similarly.

\qed

 Below we prove Theorem \ref{Th_main} in the attenuated case  for even order tensors. The proof for the odd order tensors follows similarly.

Let  $m$ be an even integer and attenuation  $a\in C^{m+1,\mu}(\ol\OM), \mu >1/2$.

In our inverse problem, the $k$-level flux solution $\ds  u^k, \; 0 \leq k \leq m$ of the boundary value problem \eqref{bvp_UK_transport} is unknown in $\OM$, since $m$-tensors  $\bbf$ is unknown. However, their traces  	
\begin{align}\label{trace:agk}
	g^k=
	\begin{cases} 
		u^k \lvert_{\Gam_{+}} &\mbox{ on } \Gam_{+},\\
		0 &\mbox{ on } \Gam_{-},
	\end{cases}
\end{align}
are known on $\Gam \times \sph$ 
from the momenta data $\langle I^0_a\bbf, I^1_a\bbf,  \cdots , I^m_a\bbf \rangle$ via \eqref{eq:gk-Ik}:
\begin{align}\label{data_equiv_ART}
	\langle g^0, g^1,   \cdots ,  g^m \rangle \longleftrightarrow \langle I_a^0\bbf, I_a^1\bbf,  \cdots , I_a^m\bbf \rangle.
\end{align}

While unknown, by Proposition \ref{prop_transEq} the a priori smoothness assumption on $m$-tensors  $\bbf$ and  attenuation $a$ yield 
$\ds u^k \in H^{m+\frac{3}{2}}(\Omega\times\sph)$. Thus $ \ds g^k, I^k_a\bbf \in H^{m+\frac{3}{2}}(\sph; H^{m+\frac{1}{2}}(\Gam) )$ for $0 \leq k \leq m$.

For $0\leq k\leq m$, let $\bu^k =  \langle u^k_{0}, u^k_{-1}, u^k_{-2} \cdots \rangle$ be the sequence valued map of the Fourier coefficients of  $u^k$ in the angular variable,  and $\bg^k= \bu^k \lvert_{\Gam}$ be its corresponding trace on the boundary.

 By identifying the same order modes in \eqref{bvp_UK_transport},  $\bu^k$ solves
\begin{subequations}\label{Beltrami_uk_sys_ART}
	\begin{align} \label{u0_1_f0Eq_ART}
		\ol{\del} \overline{u^0_{-1}}+ \del u^0_{-1} +au^0_{0} &= f_{0}, \\   \label{Beltrami_u0_Feq_ART}
	\dba\bu^0 +L^2 \del\bu^0 + a L\bu^0 &= L \bF, \\ \label{Beltrami_ukseq_Eq_ART}
	\dba\bu^k +L^2 \del\bu^k + a L\bu^k&= L \bu^{k-1}, \quad 1 \leq k \leq m, 
\end{align}
	subject to 
	\begin{align}\label{boldgk_uk_ART}
		\bg^{k}&=\bu^{k} \lvert_{\Gam }, \quad \text{for } \, 0 \leq k \leq m,
	\end{align}	
\end{subequations} 
where $\bF$ is as defined in \eqref{eq:bF}.

  The existence of the solution to the boundary value problem \eqref{Beltrami_uk_sys_ART} is postulated by the forward problem.  By Proposition \ref{prop_transEq}, 
the $k$-level flux   $u^k 	\in H^{m+\frac{3}{2}}(\sph; H^{m+1}(\OM))$. Thus, the sequences
$\bu^k \in l^{2,m+\frac{3}{2}}(\BN;H^{m+1}(\OM))$ and $\bg^k \in l^{2,m+\frac{3}{2}}(\BN;H^{m+\frac{1}{2}}(\Gam))$.

For $e^{-G}$ as in \eqref{eGop_leftshift}, define
\begin{align}\label{eq:vkuk}
\bv^k := e^{-G} \bu^k, \quad \text{for } 0\leq k \leq m.
\end{align}

By  Lemma \ref{beltrami_reduction}, $\bv^k \in l^{2,m+1}(\BN;H^{m+1}(\OM))$, $\bv^{k} \lvert_{\Gam } \in l^{2,m+1}(\BN;H^{m+1/2}(\Gam)) $, and $\bv^k$ solves 
\begin{subequations}\label{Beltrami_vk_sys_ART}
	\begin{align} \label{v0_1_f0Eq_ART}
		\ol{\del} \overline{v^0_{-1}}+ \del v^0_{-1}  &= \left( e^{-G}\bF \right)_{0}, \\    \label{Beltrami_v0_Feq_ART}
	\dba\bv^0 +L^2 \del\bv^0 &= L [e^{-G}\bF], \\ \label{Beltrami_vkseq_Eq_ART}
	\dba\bv^k +L^2 \del\bv^k &= L \bv^{k-1}, \quad 1 \leq k \leq m, 
\end{align}
	subject to 
	\begin{align}\label{boldgk_vk_ART}
		\bv^{k} \lvert_{\Gam }&=e^{-G}\bg^{k}, \quad \text{for } \, 0 \leq k \leq m.
	\end{align}	
\end{subequations}

Note that $L^{m+1} \bF = \bzero= \langle 0,0,...\rangle$. Since $e^{ \pm G}$ and $L$ commute,  
$$\ds L^{m+1} [e^{-G}\bF] = e^{-G} L^{m+1}\bF =  e^{-G}\bzero= \bzero.$$
Thus, 
$L^{m-k}\bv^k$ solve
\begin{subequations}\label{Shift_Beltrami_vk_sys_ART}
\begin{align} \label{Shift_Beltrami_v0_Feq_ART}
	\dba[L^{m}\bv^0] +L^2 \del[L^{m}\bv^0] &= \bzero, \\ \label{Shift_Beltrami_vkseq_Eq_ART}
	\dba[L^{m-k}\bv^k] +L^2 \del[L^{m-k}\bv^k] &= L^{m-k+1}\bv^{k-1}, \quad 1 \leq k \leq m, 
\end{align}
	subject to 
	\begin{align}\label{boldLmgk_vk_ART}
  L^{m-k}\bv^{k} \lvert_{\Gam }=e^{-G}L^{m-k}\bg^{k}, \quad  	0 \leq k \leq m.
	\end{align}	
\end{subequations}


Since the attenuation $a$ is known and the sequences  in \eqref{boldgk_uk_ART}  are known on $\Gam$ from \eqref{data_equiv_ART}, the sequences $ e^{-G} \bg^k$ in \eqref{boldgk_vk_ART} are also
determined. 

Proposition \ref{prop_estimate_delvk} with $\bg$ replaced by $e^{-G}\bg$ therein yields:
\begin{prop}\label{prop_estimate_delvk_ART}
		For $a\in C^{m+1,\mu}(\ol\OM), \mu >1/2$ given, let $ \langle \bg^0, \bg^1,   \cdots , \bg^m\rangle $ be the data as in \eqref{boldgk_uk_ART} obtained for some unknown even order $m$-tensor in $ H^{m+\frac{3}{2}}_0(\mathbf{S}^m; \OM)$.
 Then $\bg^k \in l^{2,m+\frac{3}{2}}(\BN;H^{m+\frac{1}{2}}(\Gam))$, and the  unique solution $\ds \bv^k$  of the boundary value problem \eqref{Beltrami_vk_sys_ART} 
satisfies
		\begin{align*}
			L^{m-k}	\bv^k (z)&=  \sum_{j=0}^{k} \BT^j L^{m-k+j}  [\B e^{-G}\bg^{k-j}](z), \quad z \in \OM, \; 0 \leq k \leq m,
		\end{align*}
    where   $\B,\BT$ are the operators in \eqref{BukhgeimCauchyFormula}, respectively \eqref{T_Formula}, and $e^{\pm G}$ are the operators  in \eqref{eGop_leftshift}. Moreover,
	\begin{equation*}
	\begin{aligned}
		\lnorm{ L^{m-k}   \bv^k}_{m-k,k+1}^2 	&\lesssim 
		\sum_{j=0}^{k}	\lnorm{ e^{-G} L^{m-j}  \bg^{j}}^2_{m+\frac{3}{2},j+\frac{1}{2}}, \quad 0 \leq k \leq m.
	\end{aligned} 
\end{equation*}

\end{prop}

The reconstruction method of the non-attenuated case  recovers
 $e^{-G} \bF$  via \eqref{v0_1_f0Eq_ART} and \eqref{Beltrami_v0_Feq_ART}:
\begin{align*}
    \left( e^{-G}\bF \right)_{0}&:= 2 \re \left[ \del v^0_{-1} \right]  \quad \text{and}\quad
	L [e^{-G}\bF]:= \dba\bv^0 +L^2 \del\bv^0,
\end{align*}
with the estimate 
$\ds 	\lnorm{ e^{-G} \bF}_{0,0}^2  \lesssim 
 		\sum_{j=0}^{m}	\lnorm{ e^{-G} L^{m-j}  \bg^{j}}^2_{m+\frac{3}{2},j+\frac{1}{2}}. 	$

Since $\ds  \bF = e^{G} \left[e^{-G} \bF \right]$, an application of  
 Lemma \ref{beltrami_reduction} yields $\ds \lnorm{ \bF}_{0,0}^2  
 \lesssim 
 \sum_{j=0}^{m}	\lnorm{ L^{m-j}  \bg^{j}}^2_{m+\frac{3}{2},j+\frac{1}{2}}. $
\qed

\section{Numerical inversion of the momenta Doppler transform} \label{sec:numerics}

In this section we apply the reconstruction method 
to three numerical examples in the case $m=1$ and $a \equiv 0$. Since the odd tensor case was not detailed in Section \ref{sec:Thproof_NART}, we do it here in the Doppler case.
Specifically, a  vector field $\bbf = \langle F_1, F_2\rangle$   is to be determined from its Doppler transform
\begin{align}\label{DF} 
	I^0\bbf(x,\btheta) & =\int_{-\infty}^{\infty} 
	\btheta \cdot \bbf( \Pi_{\btheta}(x)+t\btheta)  dt, 
\end{align}
and its first-moment-Doppler transform
\begin{align}\label{D1F} 
	I^1\bbf(x,\btheta) =
	\int_{-\infty}^{\infty} t
	\btheta \cdot \bbf( \Pi_{\btheta}(x)+t\btheta)  dt, \quad (x,\btheta)\in\BR^2\times\sph,
\end{align}
where $ I^0$ and $ I^1$ are  the ray transform as in \eqref{Iaf}.

For $k =0,1$, let $v^k$ be the $k$-level flux solution  of the boundary value problem \eqref{bvp_UK_transport}. 

In our inverse problem, the solution $(v^0,v^1)$ 
is unknown in $\OM$,
since $\bbf$ is unknown. However, from the momenta Doppler data $\langle I^0 \bbf,I^1 \bbf\rangle $, 
 their traces  	
$\ds 	g^k=
	\begin{cases} 
		v^k \lvert_{\Gam_{+}} &\mbox{ on } \Gam_{+},\\
		0 &\mbox{ on } \Gam_{-},  
	\end{cases}
$
are determined via \eqref{eq:gk-Ik}:
\begin{align}\label{data_equiv_1tensor}
	\langle g^0, g^1 \rangle \longleftrightarrow \langle I^0\bbf, I^1\bbf \rangle.
\end{align}

 By identifying the Fourier coefficients in \eqref{bvp_UK_transport}, the solution $v^k_{n}$'s  solve
 \begin{subequations}\label{Beltrami_m=1_sys_NART_m=1}
\begin{align}
	\label{source_syseq1}
	&\overline{\del} v^0_{0}(z)+\del v^0_{-2}(z) = f_{1}(z), \\ \label{infinite_sys_u0f1}
	&\dba v^0_{-n}(z) +\del v^0_{-n-2}(z)  =0,\qquad \qquad n\geq 1, \\ \label{Beltrami_u1Eq}
	&\ol{\del} v^1_{-n}(z) + \del v^1_{-n-2}(z) = v^{0}_{-n-1}(z), \quad n \in \BZ,
\end{align}      
and 
\begin{align}\label{trace:gkn2}
	v^k_{-n} \lvert_{\Gam }= g^k_{-n},  \quad k =0,1,
\end{align}
where 
\begin{align}\label{defn:f1}
	f_{1}:= \frac{ 1} {2} (F_{1}+ \i F_{2}).
\end{align} 
\end{subequations}

Since the vector field $\bbf$ is real valued, the solution $v^k$ of  \eqref{bvp_UK_transport}  is also  real valued, and its Fourier modes  $v^k_{-n}$'s in the angular variable occur in conjugates:
\begin{align}\label{reality_vkn}
	v^k_{n} = \ol{v^k_{-n}}, \quad \text{ for }  \; n \geq 0, \; k=0,1. 
\end{align}	
Thus, it suffices to consider the non-positive Fourier modes of $v^k$.

As before, let $\bv^k $  be the sequence valued map of the  non-positive Fourier coefficients of 
 $v^k$,  and 	$\ds \bg^k = \bv^k \lvert_{\Gam} $ be its trace. 
Then in the sequence valued map notation the boundary value problem \eqref{Beltrami_m=1_sys_NART_m=1} becomes
\begin{subequations}\label{Beltrami_m=1_sys_NART_mOdd}
	\begin{align}  \label{source_syseq11}
		\overline{\del} v^0_{0}+\del v^0_{-2} &= f_{1},  \\ \label{u0_1_f0Eq} 
		\dba [L\bv^0] +L^2 \del [L \bv^0] &= \bzero , \\  \label{Beltrami_ukseq_Eq}
		\dba\bv^1 +L^2 \del\bv^1 &= L \bv^{0}, 
	\end{align}
	subject to 
	\begin{align}\label{boldgk_vk_NART_m=1}
		\bv^{k} \lvert_{\Gam }&=\bg^{k}, \quad \text{for } \, k =0, 1. 
	\end{align}	
\end{subequations}

Note that the sequences $ \langle \bv^0, \bv^1\rangle$ of the boundary value problem \eqref{Beltrami_m=1_sys_NART_mOdd}  are unknown in $\OM$, since $f_1$ is unknown. However,  the data $ \langle \bg^0, \bg^1\rangle$ in \eqref{boldgk_vk_NART_m=1} are known on $\Gam$ from Doppler data $\langle I^0 \bbf,I^1 \bbf\rangle $ via  \eqref{data_equiv_1tensor}.

\subsection{ Reconstruction.}
Given  the Doppler data $\langle I^0 \bbf,I^1 \bbf\rangle $, 
 we  recover $f_1$ in \eqref{defn:f1} and thus $\bbf$ as follows:
\begin{figure}[ht!]
	\centering
	\includegraphics[scale=0.65]{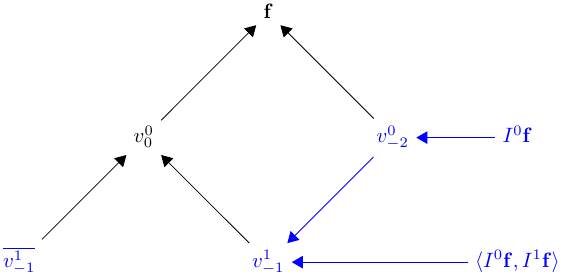}
	\caption{{Flow of the reconstruction of the vector field $\bbf$.  In the sweep down, the modes $v^0_{-2}$ and $v^1_{-1}$ colored in blue are determined layer by layer from the momenta Doppler data  $\langle I^0 \bbf,I^1 \bbf\rangle $ and the previous layer. 
			In the sweep up, also layer by layer starting from the bottom, the remaining coefficient $v^0_{0}$  is recovered. 
		}}
	\label{fig:Flow_1tensor}
\end{figure}

\begin{itemize}
	\item  {\bf Step 1 (Sweep down): The recovery of the sequence $L \bv^0 $.}

  From 	\eqref{u0_1_f0Eq}, we note that  $L\bv^0 $ is $L^2$-analytic,
and can be recovered via the Bukhgeim-Cauchy Integral formula \eqref{BukhgeimCauchyFormula}:
\begin{align}\label{uzero_mODD_B}
	L\bv^0=  \B (L\bg^0 ).
\end{align}
Componentwise,   for $n\geq 1$,
\begin{align} \label{uzero_mODD_B1}
	v^0_{-n}(z) &= \frac{1}{2\pi \i} \int_{\Gam}
	\frac{ g^0_{-n}(\zeta)}{\zeta-z}d\zeta  + \frac{1}{2\pi \i}\int_{\Gam} \left \{ \frac{d\zeta}{\zeta-z}-\frac{d \ol{\zeta}}{\ol{\zeta}-\ol{z}} \right \} \sum_{j=1}^{\infty}  
	g^0_{-n-2j}(\zeta)
	\left( \frac{\ol{\zeta}-\ol{z}}{\zeta-z} \right) ^{j},\; z\in\OM.
\end{align}

Note that the mode $v_0^0$ is not yet determined. 

\vspace{0.2cm}
	\item  {\bf Step 2 (Sweep down): The recovery of the entire  sequence $\bv^1 $.}
	\vspace{0.2cm}

Since the modes $ \langle v^0_{-1}, v^0_{-2}, v^0_{-3}, \cdots \rangle $ are now recovered in $\OM$ by \eqref{uzero_mODD_B1}, the right hand side of the non-homogeneous Bukhgeim Beltrami system \eqref{Beltrami_ukseq_Eq} is known.
The solution $\bv^1$ of  \eqref{Beltrami_ukseq_Eq} and \eqref{boldgk_vk_NART_m=1} is given by the Bukhgeim-Pompeiu formula \eqref{BP_Formula_vk}:
\begin{equation}\label{construction_v1}
	\begin{aligned}
		\bv^{1}  &=   	\B \bg^{1}+		\BT (L\bv^{0}) =   	\B \bg^{1}+		\BT [\B L\bg^0]
	\end{aligned}
\end{equation}
where the last equality uses \eqref{uzero_mODD_B}.

While the entire sequence $\bv^1$ is determined, we only need the $v^1_{-1}$ component:
  \begin{equation}\label{mODDconstruction_u1_BP}
	\begin{aligned} 
		v^1_{-1}(z) &:= \frac{1}{2\pi \i} \int_{\Gam}
		\frac{ g^1_{-1}(\zeta)}{\zeta-z}d\zeta  
		+ \frac{1}{2\pi \i}\int_{\Gam} \left \{ \frac{d\zeta}{\zeta-z}-\frac{d \ol{\zeta}}{\ol{\zeta}-\ol{z}} \right \} \sum_{j=1}^{\infty}  
		g^1_{-1-2j}(\zeta)
		\left( \frac{\ol{\zeta}-\ol{z}}{\zeta-z} \right) ^{j} \\
		& \qquad  -\frac{1}{ \pi } \sum_{j=0}^{\infty} \int_{\OM}    \frac{v^0_{-2-2j} (\zeta)  }{\zeta-z} \left( \frac{\ol{\zeta}-\ol{z}}{\zeta-z} \right) ^{j}d\xi d\eta, \qquad \zeta = \xi +\i \eta,   \quad z \in \OM.
	\end{aligned}
\end{equation}

\item {\bf Step 3 (Sweep up): The reconstruction of the Fourier mode $v^0_{0}$.}
\vspace{0.2cm}

  The real valued  Fourier mode $v^0_{0}$ is determined via equation \eqref{Beltrami_u1Eq} for $n=1$, and complex conjugate relation \eqref{reality_vkn}, by
\begin{align}\label{eq:mODD_u00}
	v^0_{0}(z) &:=  \ol \del v^1_{1}(z) +  \del v^1_{-1}(z) =2 \re \del v^1_{-1}(z), \quad z \in \OM.
\end{align}

Thus, from \eqref{uzero_mODD_B1} and \eqref{eq:mODD_u00}, the entire sequence $\bv^0 = \langle v^0_0, v^0_{-1}, \cdots \rangle$ is now determined in $\OM$.

\vspace{0.2cm}
\item	{\bf Step 4: The recovery of the vector field $\bbf$.}
\vspace{0.2cm}

From  the mode $v^0_{-2}$ in \eqref{uzero_mODD_B1} for $n=2$, 
and the mode $v^0_{0}$ in \eqref{eq:mODD_u00}, we use \eqref{source_syseq11} to recover 
\begin{align}\label{f1_defn}
	f_{1}(z) := \overline{\del} v^0_{0}(z) +\del v^0_{-2}(z), \quad z \in \OM,
\end{align}
and define the vector field inside $\OM$ by
\begin{align}\label{F_defn}
	\bbf := \langle 2 \re {f_1}, 2\im{f_1} \rangle.
\end{align}
\end{itemize}

\subsection{Numerical Implementation}
We present the results of the reconstruction in three numerical examples. To emphasize the departure from existing works recovering the solenoidal part, in the first two examples the Doppler data is simulated for two different vector fields sharing the same solenoidal part. The third example considers a rough field, with an embedded inclusion. The reconstruction from both noiseless and noisy data is performed for each example. The domain $\OM$ is  the unit disk centered at the origin and its boundary $\Gam$ is the unit circle.

Starting from a vector field $\bbf$, the data is computed by numerical integration in \eqref{DF} and \eqref{D1F} via the composite mid-point rule along lines. The data is calculated at $1{,}440$ boundary points $x\in\Gam$ of equal angular spacing, and at about $720$ equiangular outgoing directions $\btheta \in \sph$ (satisfying $x\cdot\btheta > 0$). 

To avoid an inverse crime, the reconstruction algorithm uses a different numerical path: each component of the vector field is recovered as a piecewise constant approximation on a ($1{,}750$ elements) triangular partition of $\Omega$. 
We use \eqref{uzero_mODD_B1} to compute the values of $v^0_{-n}$ at the vertices of the partition, yielding a piecewise linear approximation to $v^0_{-n}$. More precisely, if $v^0_{-n}\approx ax_1+bx_2+c$ on a triangle $\tau$,
then $\partial v^0_{-n}|_{\tau} \approx \dfrac{1}{2}(a-b\i)$ as a piecewise constant approximation required in \eqref{f1_defn}.
In contrast, the mode $v^1_{-1}$ is computed at each centroid by \eqref{mODDconstruction_u1_BP}. More precisely,
at the centroid $c$ of each triangle $\tau$,
\begin{align*}
	v^1_{-1}(c+\lambda) \approx &v^1_{-1}(c) + \partial_{x_1} v^1_{-1}(c)\lambda_1 + \partial_{x_2} v^1_{-1}(c)\lambda_2\\
	&+ \dfrac{1}{2}\partial^2_{x_1x_1} v^1_{-1}(c)\lambda_1^2 + \partial^2_{x_1x_2} v^1_{-1}(c)\lambda_1\lambda_2 + \dfrac{1}{2}\partial^2_{x_2x_2} v^1_{-1}(c)\lambda_2^2
\end{align*}
for small $\lambda=(\lambda_1,\lambda_2)$.
We write $\tau_1$, $\tau_2$, $\dotsc$, $\tau_K$ the triangles sharing vertices or edges with $\tau$. Then substituting $\lambda=c-c_k$, $k=1,2,\dotsc,K$
to the expansion, we obtain $K$ linear constraints with five unknowns $\partial_{x_1} v^1_{-1}(c)$, $\partial_{x_2} v^1_{-1}(c)$,
$\partial^2_{x_1x_1} v^1_{-1}(c)$, $\partial^2_{x_1x_2} v^1_{-1}(c)$, and $\partial^2_{x_2x_2} v^1_{-1}(c)$
required in \eqref{f1_defn} with \eqref{eq:mODD_u00} to find $\overline{\del}v^0_{0}$.
The least square method leads a unique solution to them on each $\tau$.
The boundary integrals in \eqref{uzero_mODD_B1} and \eqref{mODDconstruction_u1_BP} are approximated by the composite mid-point rule as the sum of the product of mid-point values of the integrand and arc lengths~\cite{fujiwaraSadiqTamasan20}.
Note that the singularity of the integrand of the final term
in \eqref{mODDconstruction_u1_BP} is removable 
and a conventional mid-point numerical integration rule can be applied. More precisely, we use
\begin{align*}
	\int_{\OM}    \frac{v^0_{-2-2j} (\zeta)  }{\zeta-z} \left( \frac{\ol{\zeta}-\ol{z}}{\zeta-z} \right) ^{j}d\xi d\eta
	&\approx \sum_{m} v^0_{-2-2j} (c_m) \int_{\tau_m}\frac{1}{\zeta-z} \left( \frac{\ol{\zeta}-\ol{z}}{\zeta-z} \right) ^{j}d\xi d\eta
	\\
	&= \sum_{m} v^0_{-2-2j} (c_m) \int_{-\pi}^{\pi} \rho_m(z;\varphi) e^{-(2j+1)i\varphi} d\varphi,
\end{align*}
where $\{\tau_m\}$ gives a triangulation of $\OM$,
$c_m$ is the centroid of $\tau_m$,
and $\rho_m(z;\varphi)$ is the length of the half-line starting from $z$ in the $\varphi$-direction cut by the triangle $\tau_m$;
if $z$ belongs to $\tau_m$, then $\tau_m$ is the distance between $z$ and $\partial\tau_m$ in the $\varphi$-direction (Fig.~\ref{fig:integral}).
The last integral above is 
computed by the eight-point Gauss-Legendre rule
with algebraically calculated $\rho_m$ as explained in the figure below.

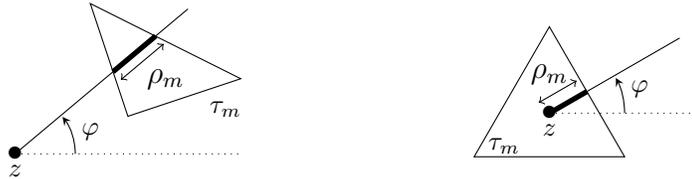
\begin{figure}[ht]
	\begin{minipage}{.12\textwidth}
		\quad
	\end{minipage}
	\begin{minipage}{.35\textwidth}
		\centering
		\begin{tikzpicture}
			\draw [dotted](0,0) node {$\bullet$} -- (3,0);
			\draw [->,>=stealth] (0,0)+(0:0.8) arc (0:38:0.8);
			\node at (1,0.3) {\footnotesize{$\varphi$}};
			\draw (0,0) node [below] {\footnotesize{$z$}} -- +(40:3);
			\draw (3,1) -- (1.5,0.5) -- (1,2) -- cycle;
			\node at (2.8,0.6) {\footnotesize{$\tau_m$}};
			
			\draw [line width=2pt] +(40:1.7) -- +(40:2.44);
			\draw [<->,xshift=3pt,yshift=-3pt] +(40:1.7) -- +(40:2.44);
			\node at (2,1) {$\rho_m$};
			
		\end{tikzpicture}
	\end{minipage}
	\begin{minipage}{.35\textwidth}
		\centering
		\begin{tikzpicture}
			\draw (0,0.58) node {$\bullet$} -- +(30:2);
			\draw [dotted] (0,0.58) node [below] {\footnotesize{$z$}} -- +(0:2);
			\draw [->,>=stealth] (0,0.58)+(0:1) arc (0:28:1);
			\node at (1.2,0.9) {\footnotesize{$\varphi$}};
			\draw (-1,0) -- (1,0) -- (0,1.73) -- cycle;
			\node at (-0.6,0.15) {\footnotesize{$\tau_m$}};
			
			\draw [line width=2pt] (0,0.58) -- +(30:0.58);
			\draw [<->,xshift=-4pt,yshift=4pt] (0,0.58) -- +(30:0.58);
			\node at (0,1.1) {$\rho_m$};
			
		\end{tikzpicture}
	\end{minipage}
	\begin{minipage}{.12\textwidth}
		\quad
	\end{minipage}
	\caption{\label{fig:integral}The length $\rho_m(z;\varphi)$ cut by the triangle $\tau_m$. The case for $z \not\in\tau_m$ (left) and that for $z \in \tau_m$ (right)}
\end{figure}

Throughout this section, the series in \eqref{uzero_mODD_B1} and 
\eqref{mODDconstruction_u1_BP} are truncated up to $256$ Fourier modes. The truncation index not only controls the accuracy, but also plays a regularizing role in stability.

In the examples below, the relative error between a reconstructed vector field $\bbf_{\text{recon}}$ and the exact  one are in the $L^2$ sense:
\begin{align}
	\norm{\bbf_{\text{recon}} -\bbf}_{\text{rel}} = \frac{\norm{\bbf_{\text{recon}} - \bbf}_2}{\|\bbf\|_2}.
\end{align}
Similarly, the relative error in the data is in the $L^2$ sense.

All numerically reconstructed results are calculated in the double precision arithmetic
on AMD EPYC 7643 with 96 threads OpenMP parallel computations.

\begin{example}\label{ex:KB}	
	We consider first the vector field 
	\begin{align}\label{eq:exampleKB}
		\bbf(x) 	&= \nabla \left( \sin\pi|x|^2 \right) + \bbf^s(x),  
	\end{align}
	where the solenoidal part 
	\begin{equation}\label{eq:Fsdefn}
		\bbf^s(x) 	= \begin{pmatrix} 2x_1x_2\cos|x|^2 + \cos (6x_1 x_2) - 6x_1x_2\sin(6x_1x_2)\\
			-\sin|x|^2 - 2x_1^2\cos|x|^2 + 6x_2^2\sin(6x_1x_2)
		\end{pmatrix};
	\end{equation}
	see \cite{kazantsevBukhgeimJr07} and Figure~\ref{fig:exampleKB:exact} below.
\end{example}
	
	{\bf {Example~\ref{ex:KB}(a) - Noiseless data:}}
	For the vector field $\bbf$ in \eqref{eq:exampleKB}, the simulated data $(I^0\bbf,I^1\bbf)$ is illustrated in Figure~\ref{fig:exampleKB:meas}, where crosses ($\times$) depict a few boundary nodes $x\in\Gam$, while the
	red and blue curves are $\{x+|I^j\bbf(x,\btheta)|\btheta \:;\: \btheta \in \sph, x\cdot\btheta > 0\}$, $j=0,1$. Also, for illustration purposes, the radial direction is shrunk by $1/5$-th.
	
	Note that $I^0\bbf,$ and $I^1\bbf,$ are not always positive. To differentiate the sign, the positive and negative parts are drawn in red, respectively in blue.
	Since only outgoing signals are measured (while the incoming flow is zero at the boundary) signals are depicted outside $\Omega$ only.
	\begin{figure}[h]
		\begin{minipage}{.39\textwidth}
			\centering
			\includegraphics[width=.65\textwidth,bb=75 20 270 210]{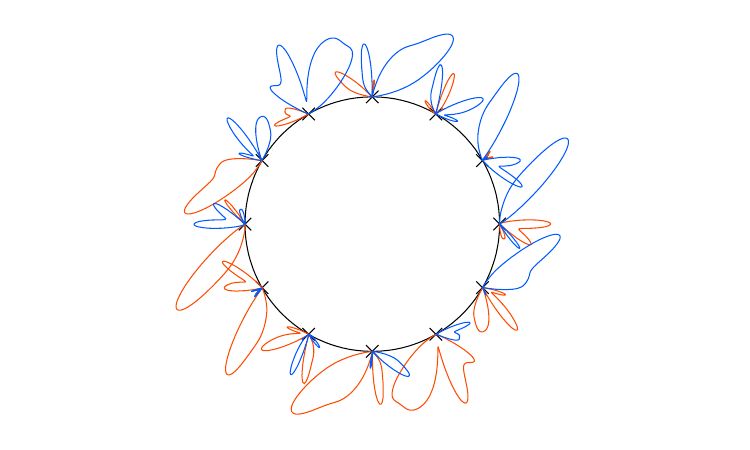}
		\end{minipage}
		\hfill
		\begin{minipage}{.2\textwidth}
			\includegraphics[width=\textwidth,bb=75 10 260 190]{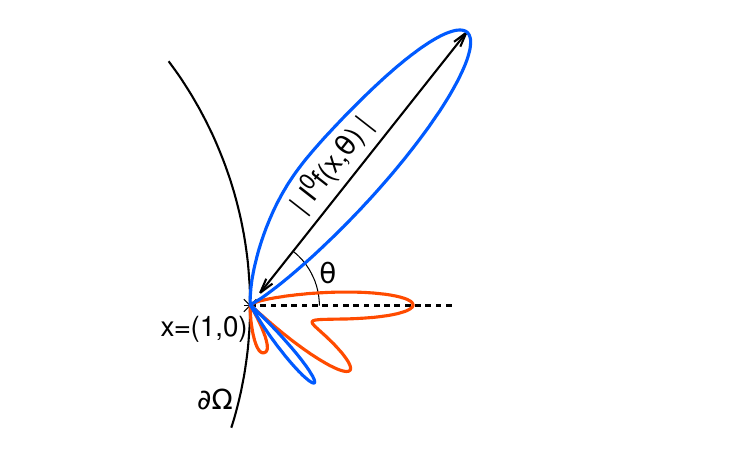}
		\end{minipage}
		\hfill
		\begin{minipage}{.39\textwidth}
			\includegraphics[width=.65\textwidth,bb=75 20 270 210]{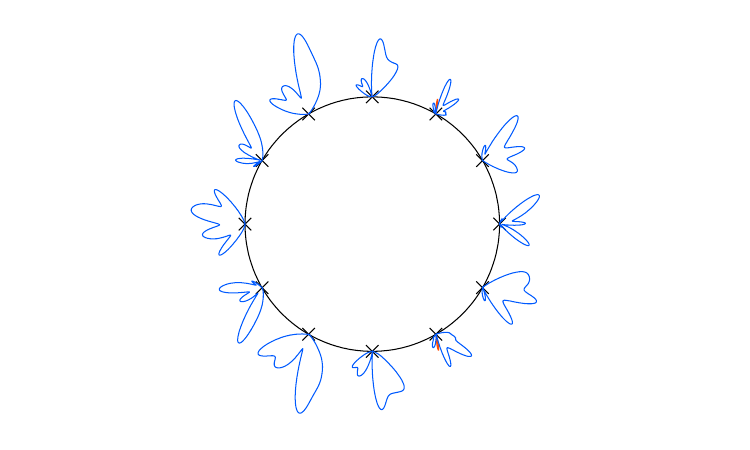}
		\end{minipage}
		\caption{\label{fig:exampleKB:meas}
			Simulated data for $\bbf$ in \eqref{eq:exampleKB}: $I^0\bbf$ (left) with its magnification at $x=(1,0)$ (middle),
			and $I^1\bbf$ (right). The crosses ($\times$) are some data collection points at the boundary, while the red and blue curves represent 
			$I^j\bbf(x,\btheta)$, $j=0,1$ in polar coordinates $\bigl(|I^j\bbf(x,\btheta)|,\btheta\bigr)$ centered at the respective boundary point $x\in\Gam$.
			The radial direction is shrunk by $1/5$-th for illustration purposes.
		}
	\end{figure}  
	
	The numerically reconstructed result shown in Figure~\ref{fig:exampleKB:tom} has a relative error of $18.1\%$. The total elapsed time in the reconstruction is approximately 10 seconds.

	\begin{figure}[ht]
		\centering
		\begin{minipage}{.3\textwidth}
			\includegraphics[width=\textwidth,bb=45 0 300 215,clip]{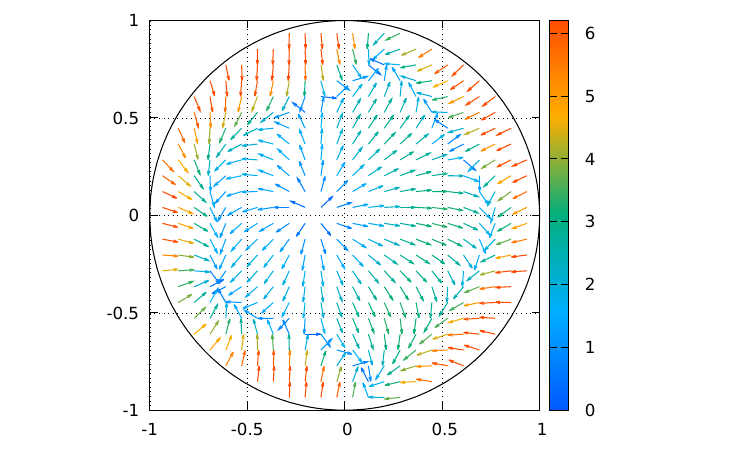}
		\end{minipage}
		\hfill
		\begin{minipage}{.33\textwidth}
			\includegraphics[width=\textwidth,bb=120 45 280 170]{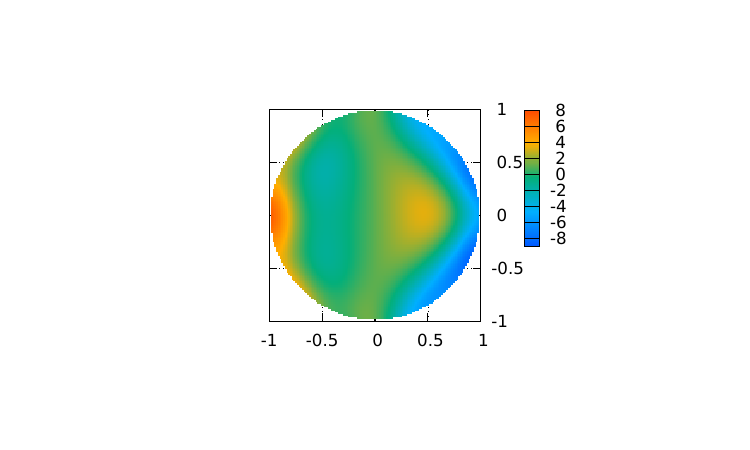}
		\end{minipage}
		\hfill
		\begin{minipage}{.33\textwidth}
			\includegraphics[width=\textwidth,bb=120 45 280 170]{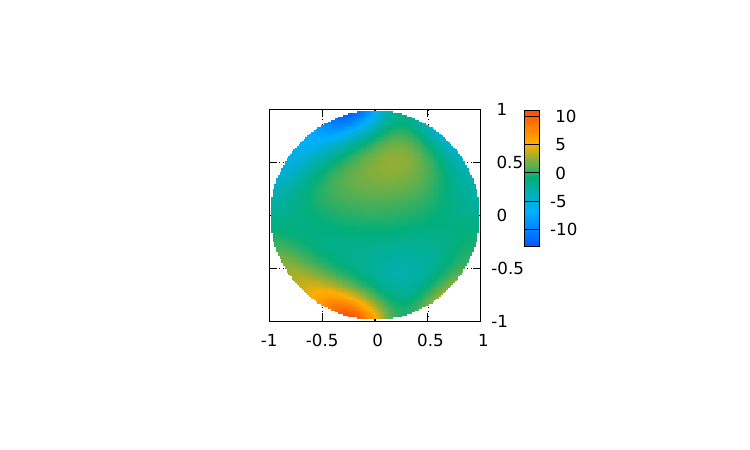}
		\end{minipage}
		\caption{\label{fig:exampleKB:exact}Exact vector field $\bbf =  \spr{F_1}{F_2}$ in \eqref{eq:exampleKB} (left), its first component $F_1$ (middle) and its second component $F_2$ (right).}
	\end{figure}
	\begin{figure}[ht]
		\begin{minipage}{.3\textwidth}
			\includegraphics[width=\textwidth,bb=45 0 300 215,clip]{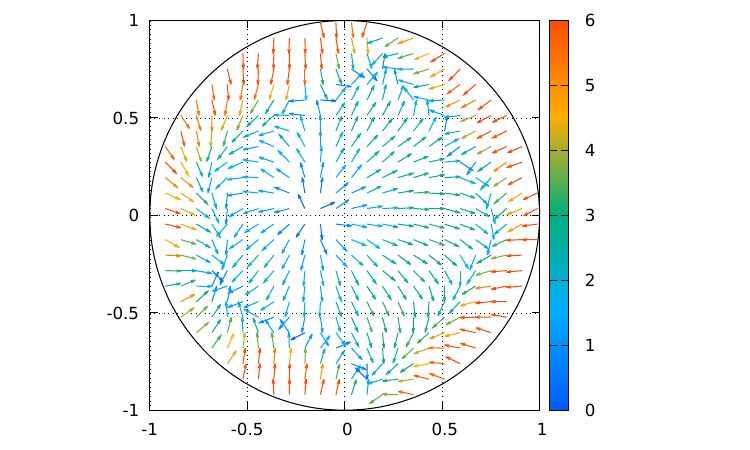}
		\end{minipage}
		\hfill
		\begin{minipage}{.33\textwidth}
			\includegraphics[width=\textwidth,bb=120 45 280 170]{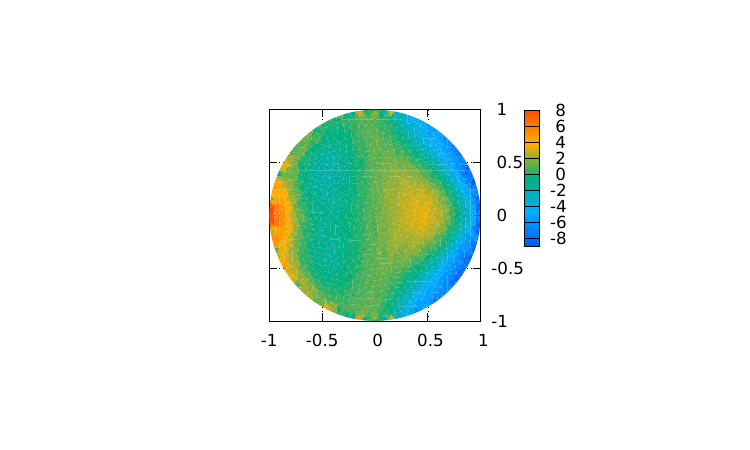}
		\end{minipage}
		\hfill
		\begin{minipage}{.33\textwidth}
			\includegraphics[width=\textwidth,bb=120 45 280 170]{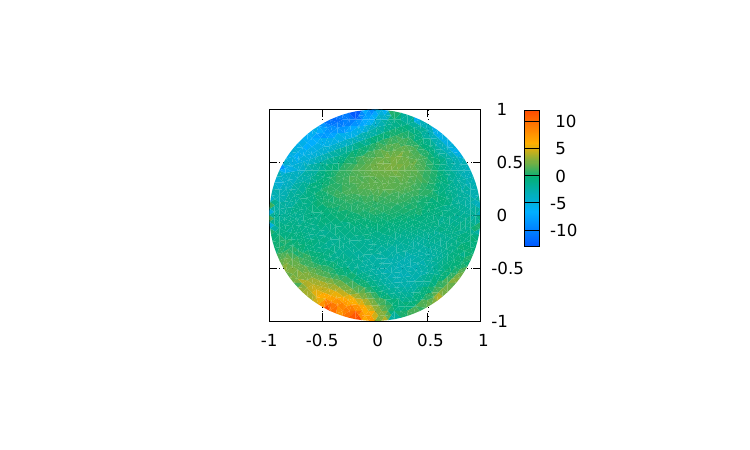}
		\end{minipage}
		\caption{\label{fig:exampleKB:tom}Numerical reconstruction from noiseless data: The vector field $\bbf_{0,\text{recon}}$ (left),
			its first component $F_1$ (middle) and its second component $F_2$ (right).
		}
	\end{figure}


{\bf {Example~\ref{ex:KB}(b) - Perturbed data within the range:}}
To illustrate the stability estimate in Theorem \ref{Th_main}, we first consider the case of data perturbed within the range. To generate such a data we solve the forward problem by \eqref{DF} and \eqref{D1F} for a perturbed vector field $\bbf_\epsilon = \bbf + \epsilon$, for some smooth vector field $\epsilon$ in $\Omega$. 
Figure~\ref{fig:exampleKB:range} below shows the reconstruction $\bbf_{\epsilon,\text{recon}}$ from this data. In this example, the relative error in the data for $I^0\bbf$ is  $5.52\%$ and
for $I^1\bbf$ is $4.48\%$, while the relative error in the reconstruction is $30.0\%$. 
\begin{figure}[ht]
	\begin{minipage}{.3\textwidth}
		\includegraphics[width=\textwidth,bb=45 0 300 215,clip]{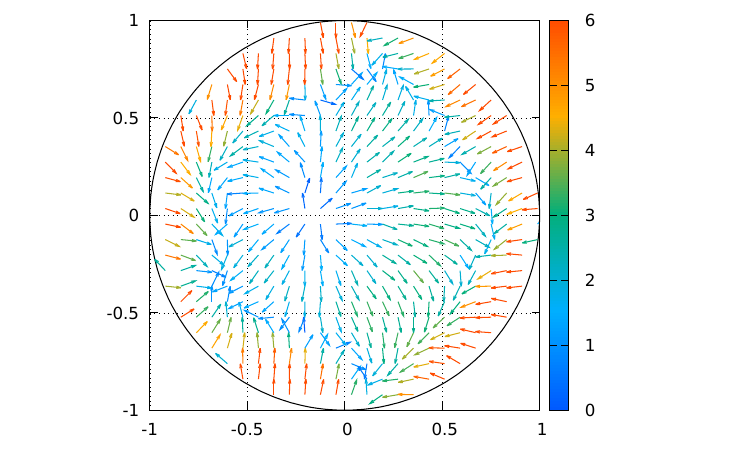}
	\end{minipage}
	\hfill
	\begin{minipage}{.33\textwidth}
		\includegraphics[width=\textwidth,bb=120 45 280 170]{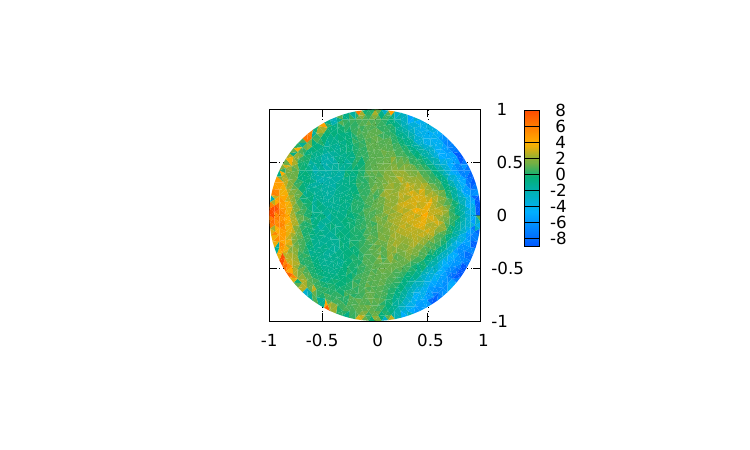}
	\end{minipage}
	\hfill
	\begin{minipage}{.33\textwidth}
		\includegraphics[width=\textwidth,bb=120 45 280 170]{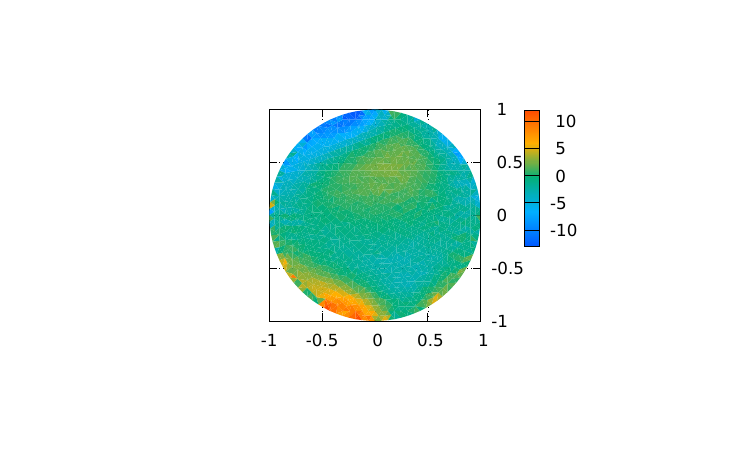}
	\end{minipage}
	\caption{\label{fig:exampleKB:range}
		Numerical reconstruction from perturbed data within the range
		($5.52\%$ relative error in $I^0\bbf$ and $4.48\%$ in $I^1\bbf$). The reconstructed field $\bbf_{\epsilon,\text{recon}} = \langle F_1, F_2 \rangle $ (left) and its components $F_1$ (middle) and $F_2$ (right) has $30.0\%$ relative error.}
\end{figure}  



{\bf {Example~\ref{ex:KB}(c) - Noisy data :}}	
To assess the robustness of the method, we consider the same vector field as in \eqref{eq:exampleKB},
where the data is corrupted with an additive random error. Specifically,  $I^0\bbf$ now contains 
about  $5.88\%$ relative error, while $I^1\bbf$ contains 
$4.34\%$ relative error, which are at the same level as in the previous example; see 
Figure~\ref{fig:exampleKB:witherror:meas} for an illustration.
\begin{figure}[h]
	\begin{minipage}{.48\textwidth}
		\centering
		\includegraphics[width=.55\textwidth,bb=75 20 270 210]{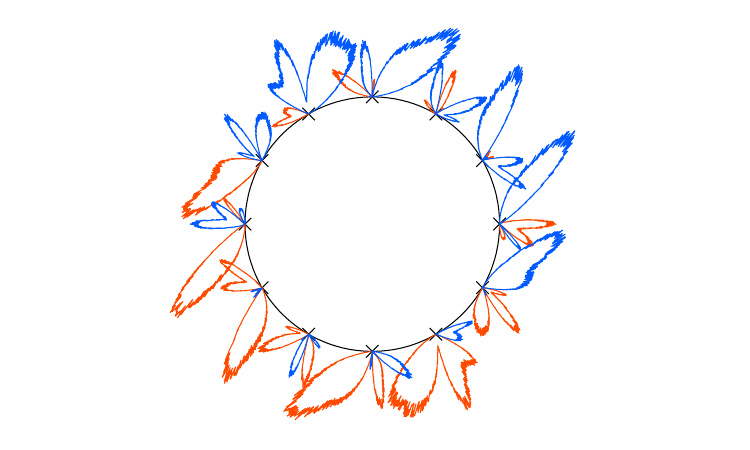}
	\end{minipage}
	\hfill
	\begin{minipage}{.48\textwidth}
		\includegraphics[width=.55\textwidth,bb=75 20 270 210]{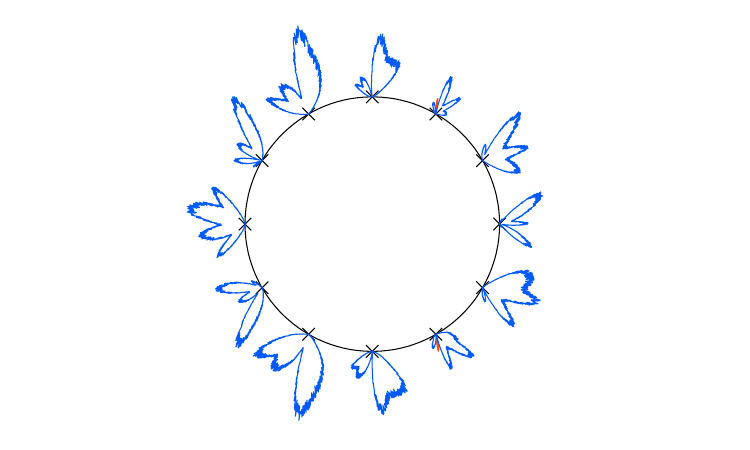}
	\end{minipage}
	\caption{\label{fig:exampleKB:witherror:meas}Noisy data $I^0\bbf$ (left) with $5.88\%$  error and $I^1\bbf$ (right) with $4.34\%$ error.}
\end{figure}  

The reconstructed vector field $\bbf_{\text{recon}}$ shown in Figure~\ref{fig:exampleKB:witherror}, contains approximately $54.6\%$ relative error.

\begin{figure}[ht]
\begin{minipage}{.3\textwidth}
	\includegraphics[width=\textwidth,bb=45 0 300 215,clip]{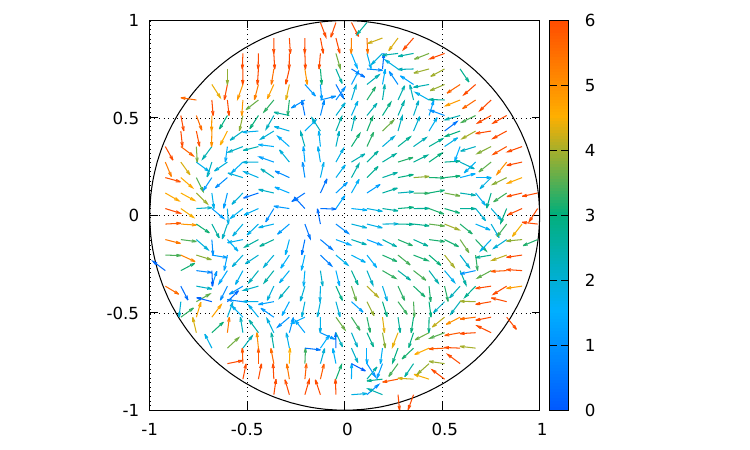}
\end{minipage}
\hfill
\begin{minipage}{.33\textwidth}
	\includegraphics[width=\textwidth,bb=120 45 280 170]{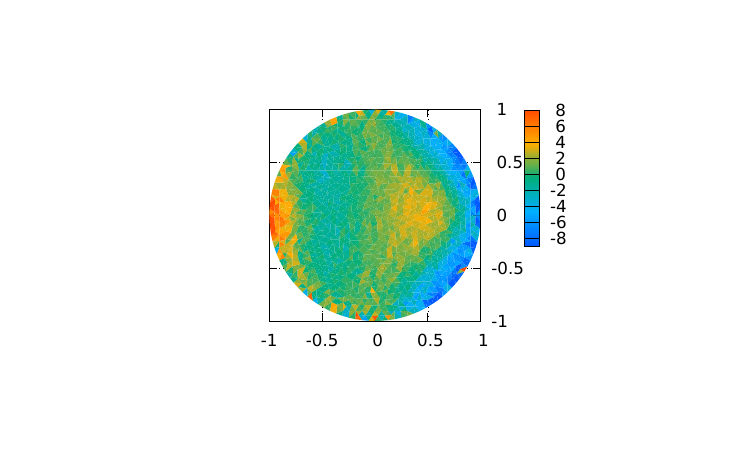}
\end{minipage}
\hfill
\begin{minipage}{.33\textwidth}
	\includegraphics[width=\textwidth,bb=120 45 280 170]{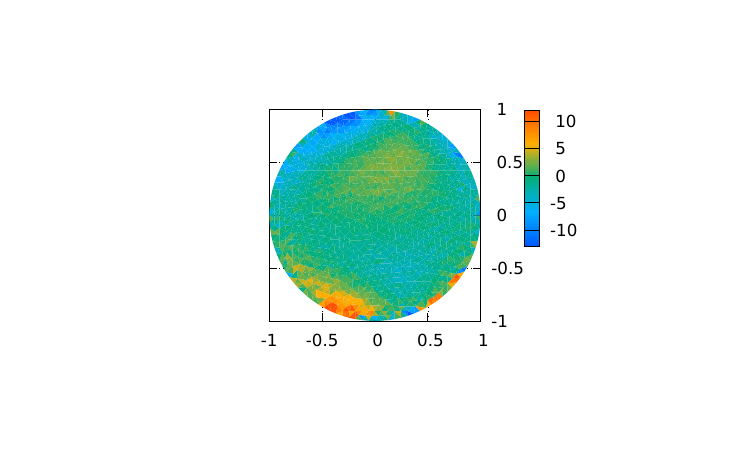}
\end{minipage}
\caption{\label{fig:exampleKB:witherror}
	Numerical reconstruction from noisy data ($5.88\%$ relative error in $I^0\bbf$ and  $4.34\%$ in $I^1\bbf$). The reconstructed field $\bbf_{\text{recon}} = \langle F_1, F_2 \rangle $ (left) and its components $F_1$ (middle) and $F_2$ (right) has $54.6\%$ relative error.}
\end{figure}  

\begin{example}
\label{ex:Atan}
We consider next  the vector field
\begin{align} \label{eq:exampleAtan}
	\bbf(x) = \nabla \left( \arctan \frac{x_2}{2+x_1} \right ) +  \bbf^s(x) 
\end{align}
with the same solenoidal part $\bbf^s$ as in \eqref{eq:Fsdefn}.
\end{example}

\begin{figure}[h]
	\begin{minipage}{.48\textwidth}
		\centering
		\includegraphics[width=.5\textwidth,bb=75 20 270 210]{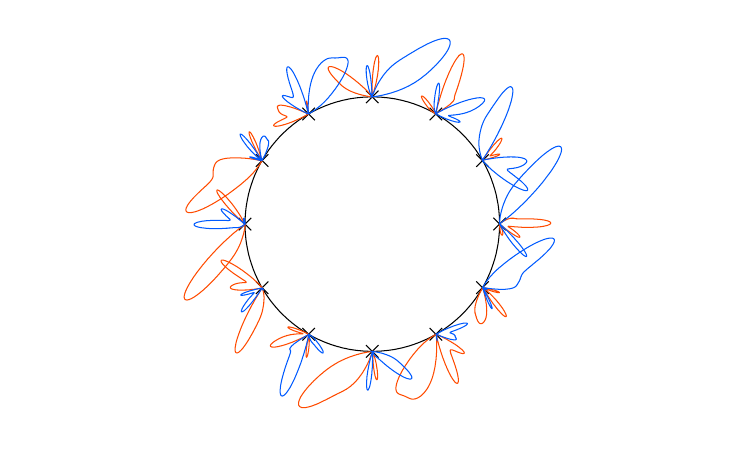}
	\end{minipage}
	\hfill
	\begin{minipage}{.48\textwidth}
		\includegraphics[width=.5\textwidth,bb=75 20 270 210]{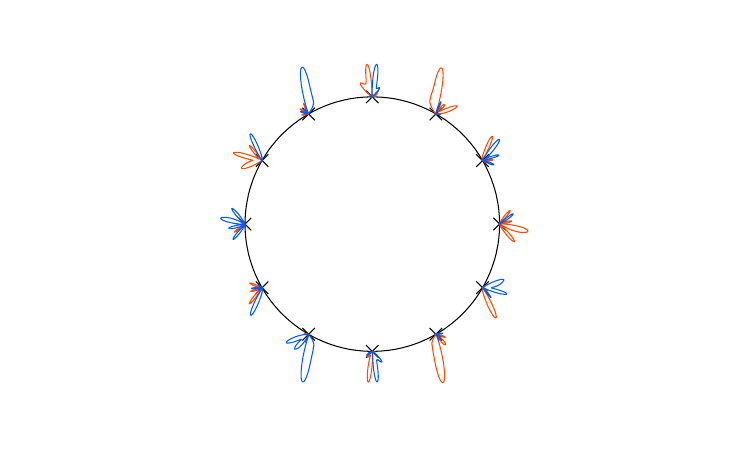}
	\end{minipage}
	\caption{\label{fig:exampleAtan:meas}
		Simulated noiseless data $I^0\bbf$ (left) and $I^1\bbf$ (right) for $\bbf$ in \eqref{eq:exampleAtan}.
		The crosses ($\times$) are some data collection points at the boundary, while the red and blue curves represent 
		$I^j\bbf(x,\btheta)$, $j=0,1$ in polar coordinates $\bigl(|I^j\bbf(x,\btheta)|,\btheta\bigr)$ centered at the respective boundary point $x\in\Gam$.
	}
\end{figure}  
\begin{figure}[h]
	\centering
	\begin{minipage}{.3\textwidth}
		\includegraphics[width=\textwidth,bb=45 0 300 215,clip]{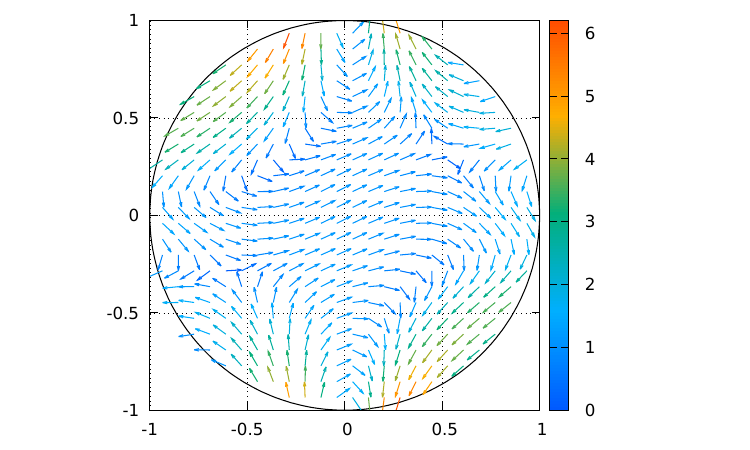}
	\end{minipage}
	\hfill
	\begin{minipage}{.33\textwidth}
		\includegraphics[width=\textwidth,bb=120 45 280 170]{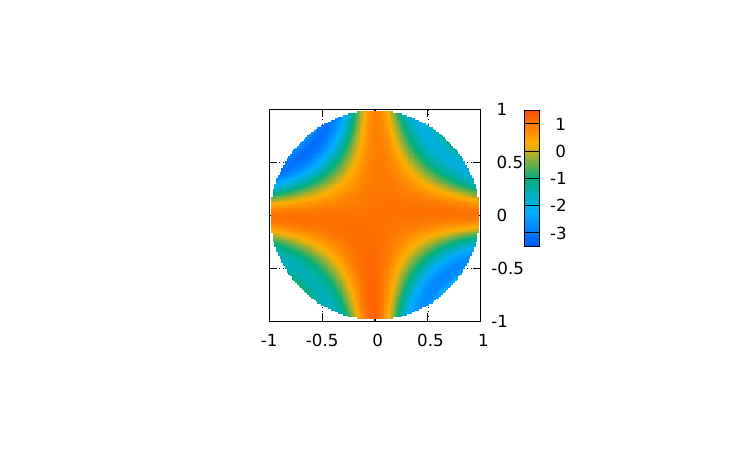}
	\end{minipage}
	\hfill
	\begin{minipage}{.33\textwidth}
		\includegraphics[width=\textwidth,bb=120 45 280 170]{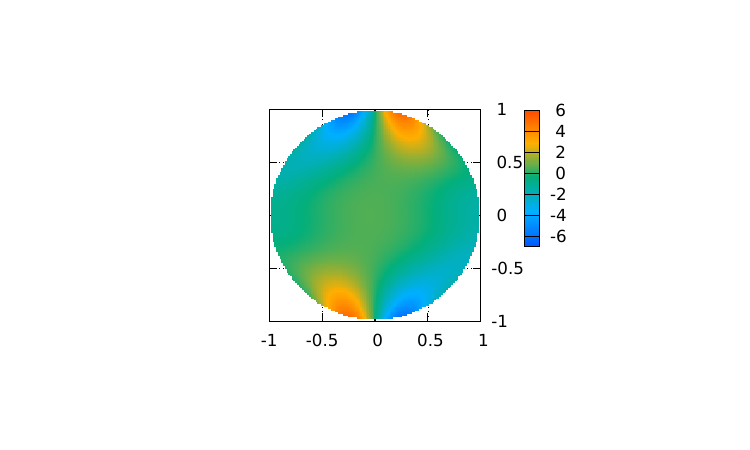}
	\end{minipage}
	\caption{\label{fig:exampleAtan:exact}Exact vector field $\bbf =  \spr{F_1}{F_2}$ in  \eqref{eq:exampleAtan} (left),  its first component $F_1$ (middle) and its second component $F_2$ (right).}
\end{figure}
\begin{figure}[h]
	\begin{minipage}{.3\textwidth}
		\includegraphics[width=\textwidth,bb=45 0 300 215,clip]{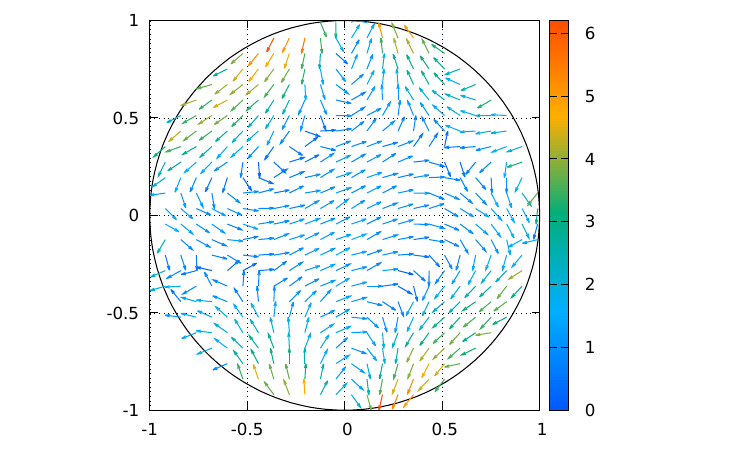}
	\end{minipage}
	\hfill
	\begin{minipage}{.33\textwidth}
		\includegraphics[width=\textwidth,bb=120 45 280 170]{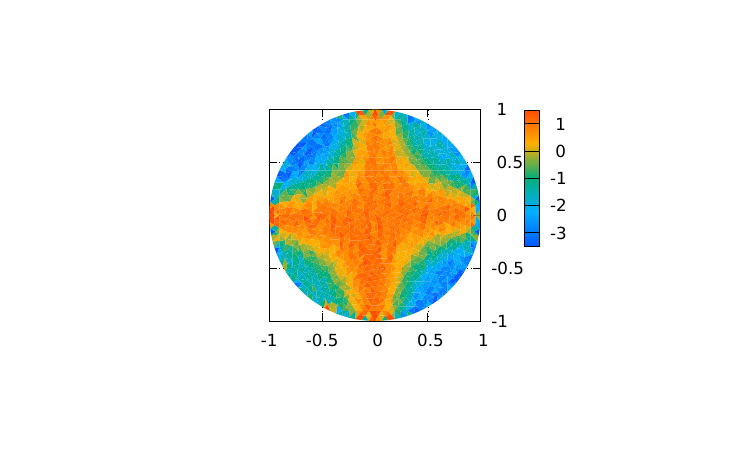}
	\end{minipage}
	\hfill
	\begin{minipage}{.33\textwidth}
		\includegraphics[width=\textwidth,bb=120 45 280 170]{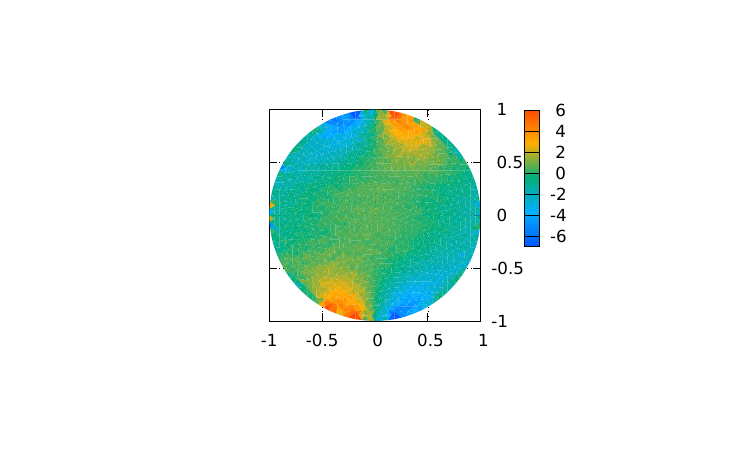}
	\end{minipage}
	\caption{\label{fig:exampleAtan}Numerically reconstructed vector field from noiseless Doppler data:
		vector field $\bbf_{0,\text{recon}} = \langle F_1, F_2 \rangle $ (left) and its components $F_1$ (middle) and $F_2$ (right).}
\end{figure}  

{\bf {Example~\ref{ex:Atan}(a) - Noiseless data:}}
The  vector field in \eqref{eq:exampleAtan} is
depicted in Figure~\ref{fig:exampleAtan:exact},  while its corresponding simulated Doppler data is shown in Figure~\ref{fig:exampleAtan:meas}. The numerically reconstructed vector field and its components  are exhibited in Figure~\ref{fig:exampleAtan} having $31.1\%$  relative error.

{\bf {Example~\ref{ex:Atan}(b) - Perturbed data within the range:}}
In this case, the generated perturbed data $\bbf_\epsilon$ depicted in Figure~\ref{fig:exampleAtan:range:meas}
has $5.02\%$ relative error in $I^0\bbf$ and
$6.04\%$ relative error in $I^1\bbf$. The numerical reconstruction of $\bbf_{\epsilon,\text{recon}}$ in Figure~\ref{fig:exampleAtan:range}
has $45.9\%$ relative error.
\begin{figure}[h]
\begin{minipage}{.48\textwidth}
	\centering
	\includegraphics[width=.55\textwidth,bb=75 20 270 210]{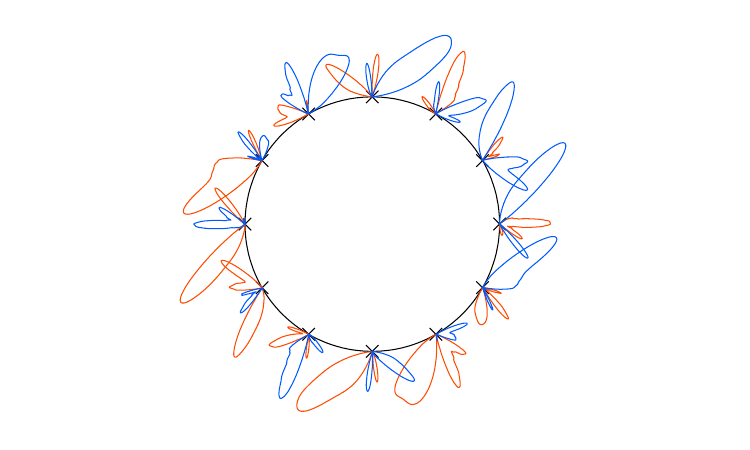}
\end{minipage}
\hfill
\begin{minipage}{.48\textwidth}
	\includegraphics[width=.55\textwidth,bb=75 20 270 210]{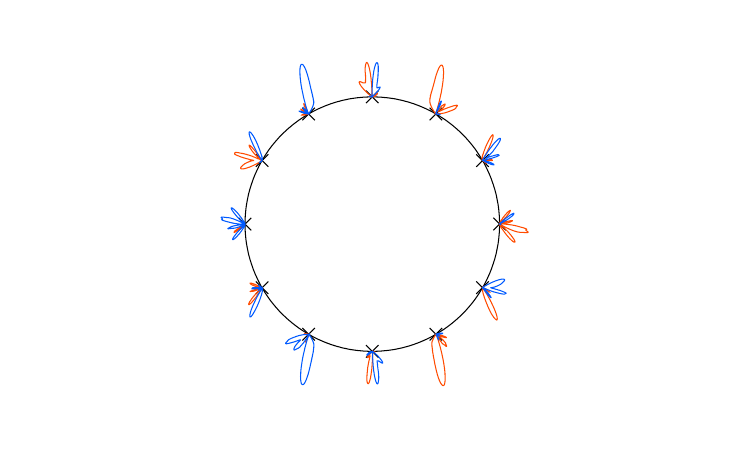}
\end{minipage}
\caption{\label{fig:exampleAtan:range:meas}Data $I^0\bbf_\epsilon$ (left) with $5.02\%$ error and $I^1\bbf_\epsilon$ (right) with $6.04\%$ error, which is considered as measurement data with noise in range}
\end{figure}  
\begin{figure}[ht]
\begin{minipage}{.3\textwidth}
	\includegraphics[width=\textwidth,bb=45 0 300 215,clip]{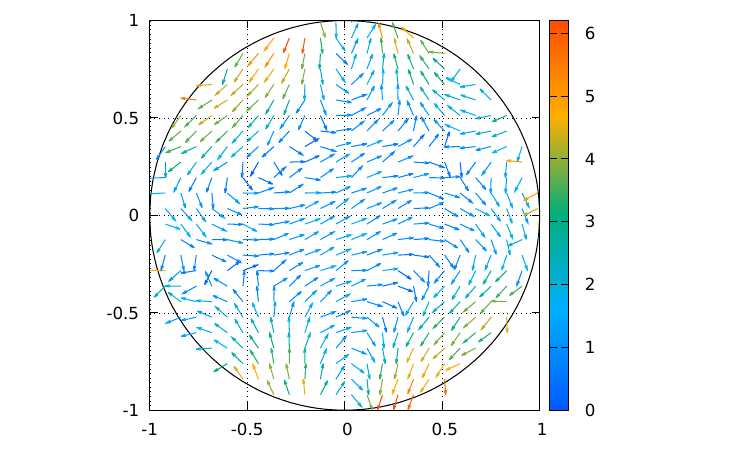}
\end{minipage}
\hfill
\begin{minipage}{.33\textwidth}
	\includegraphics[width=\textwidth,bb=120 45 280 170]{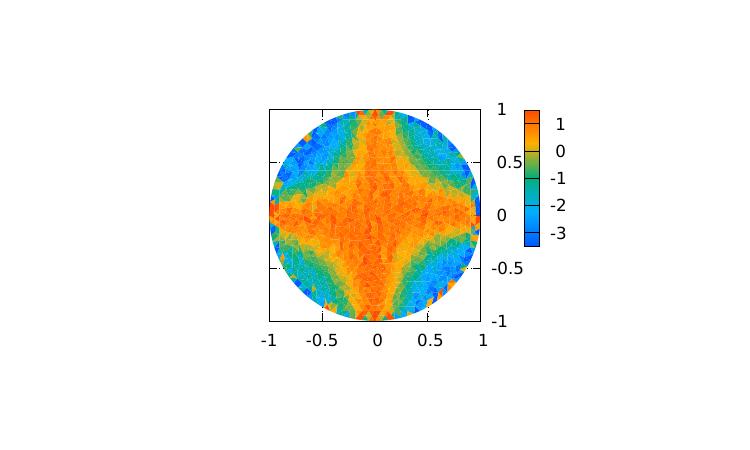}
\end{minipage}
\hfill
\begin{minipage}{.33\textwidth}
	\includegraphics[width=\textwidth,bb=120 45 280 170]{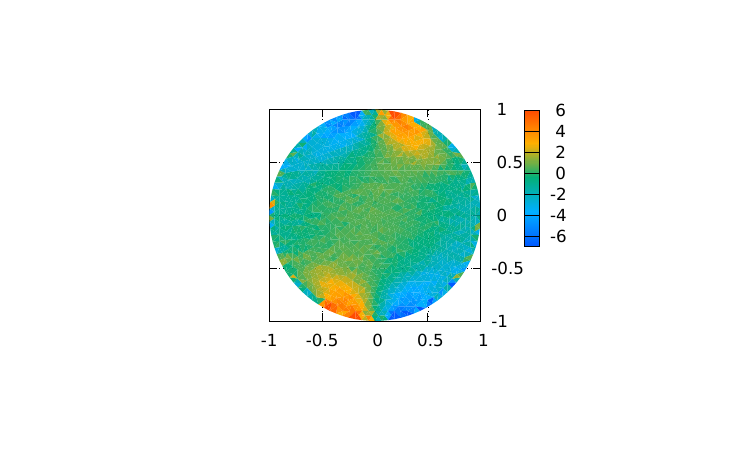}
\end{minipage}
\caption{\label{fig:exampleAtan:range}
	Numerical reconstruction from perturbed data within the range  ($5.02\%$ relative error in $I^0\bbf$ and $6.04\%$ in $I^1\bbf$).  The reconstructed field $\bbf_{\epsilon,\text{recon}} = \langle F_1, F_2 \rangle $ (left) and its components $F_1$ (middle) and $F_2$ (right) has $45.9\%$ relative error.}
\end{figure}  

{\bf {Example~\ref{ex:Atan}(c) - Noisy data:}}
We consider the same vector field as in \eqref{eq:exampleAtan},
however  the data is corrupted with additive random  errors:
$6.03\%$ relative error in $I^0\bbf$, and 
$5.04\%$  in $I^1\bbf$. 

The numerical reconstruction results are shown in   Figure~\ref{fig:exampleAtan:witherror}. The reconstructed vector field $\bbf_{\text{recon}}$ has $71.3\%$ relative  error.

\begin{figure}[h]
\begin{minipage}{.3\textwidth}
	\includegraphics[width=\textwidth,bb=45 0 300 215,clip]{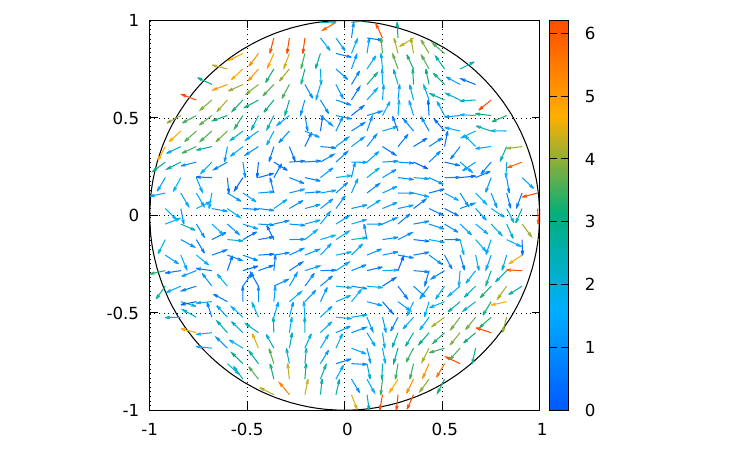}
\end{minipage}
\hfill
\begin{minipage}{.33\textwidth}
	\includegraphics[width=\textwidth,bb=120 45 280 170]{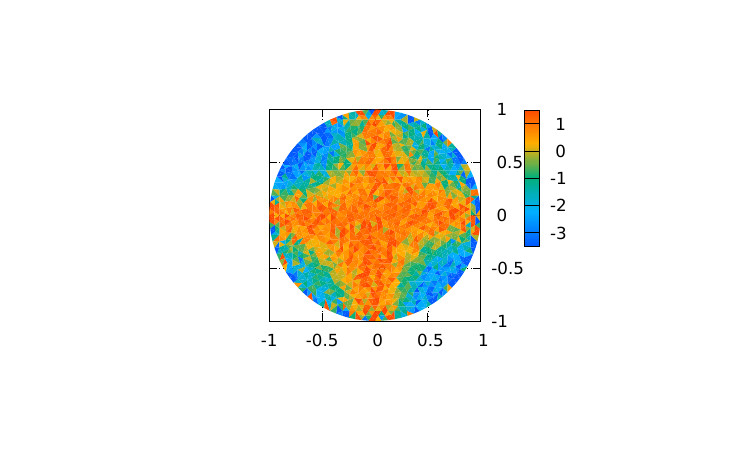}
\end{minipage}
\hfill
\begin{minipage}{.33\textwidth}
	\includegraphics[width=\textwidth,bb=120 45 280 170]{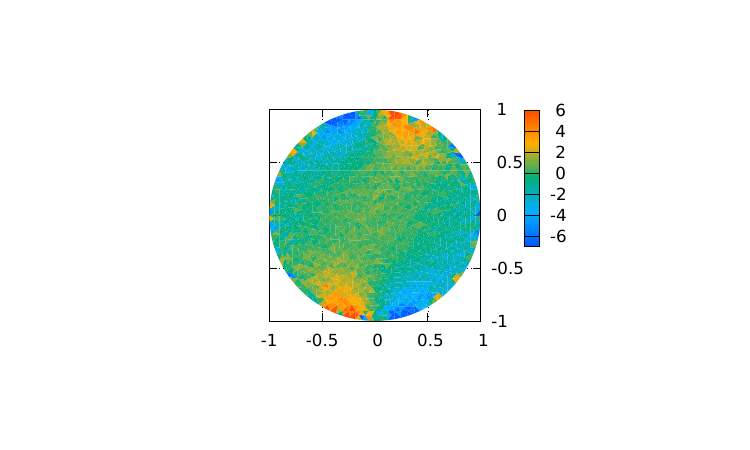}
\end{minipage}
\caption{\label{fig:exampleAtan:witherror}
	Numerical reconstruction from noisy data ($6.03\%$ relative error in $I^0\bbf$ and $5.04\%$  in $I^1\bbf$). The reconstructed field $\bbf_{\text{recon}} = \langle F_1, F_2 \rangle $ (left) and its components $F_1$ (middle) and $F_2$ (right) has $71.3\%$ relative error.}
\end{figure}  

In Table \ref{tbl:KBerrors} below, we summarize the level of error obtained in the examples.  The reconstruction error in Example \ref{ex:KB}(b) ($30.0\%$), respectively, Example \ref{ex:Atan}(b) ($45.9\%$) obtained from the perturbed data within the range reflects the instability of our method due to twice differentiation. The reconstruction error 
in Example \ref{ex:KB}(c) ($54.6\%$), respectively, Example \ref{ex:Atan}(c) ($71.3\%$) obtained from (an additive random error) noisy data is also due to the ill-posedness (non-existence) specific to inverting data outside the range.
\begin{table}[h]
\caption{\label{tbl:KBerrors}Differences between noiseless data, perturbed data in the range and  additive random noise in the first two examples.}
\centering
\begin{tabular}{l|c|c|c|c|c|c}
	\toprule
	Relative error in
	& 	Ex~\ref{ex:KB} (a) 
	& 	Ex~\ref{ex:KB} (b)
	 & Ex~\ref{ex:KB} (c) 
	 & 	Ex~\ref{ex:Atan} (a)
	& 	Ex~\ref{ex:Atan} (b) & Ex~\ref{ex:Atan} (c) 
	\\	
	\midrule
	\midrule
	$I^0\bbf$  &$0\%$ & $5.52\%$  & $5.88\%$ &$0\%$  & $5.02\%$  & $6.03\%$ \\
	\midrule
	$I^1\bbf$   &$0\%$& $4.48\%$  & $4.34\%$  
	&$0\%$  	& $6.04\%$    & $5.04\%$\\
	\midrule
	Reconstruction  &$18.1\%$ & $30.0\%$  & $54.6\%$ 
	& $31.1\%$& $45.9\%$  & $71.3\%$\\
	\bottomrule  
\end{tabular}
\end{table}


\begin{example}
\label{ex:SQ}
Let $S = [-0.25,0.75]\times[-0.5,0.5]$ be an off-centered square and
denote by $\chi_S$ its characteristic function.
Then we consider a vector field supported by $S$:
\begin{equation} \label{eq:exampleSQ}
  \bbf(x) = \chi_S(x)
  \begin{pmatrix}
    (1-(2x_1-0.5)^2)(1-4x_2^2) \\
    \cos ((2x_1+0.5)\pi)\cos(2\pi x_2)
  \end{pmatrix}.
\end{equation}

\begin{figure}[h]
	\begin{minipage}{.48\textwidth}
		\centering
		\includegraphics[width=.55\textwidth,bb=75 20 270 210]{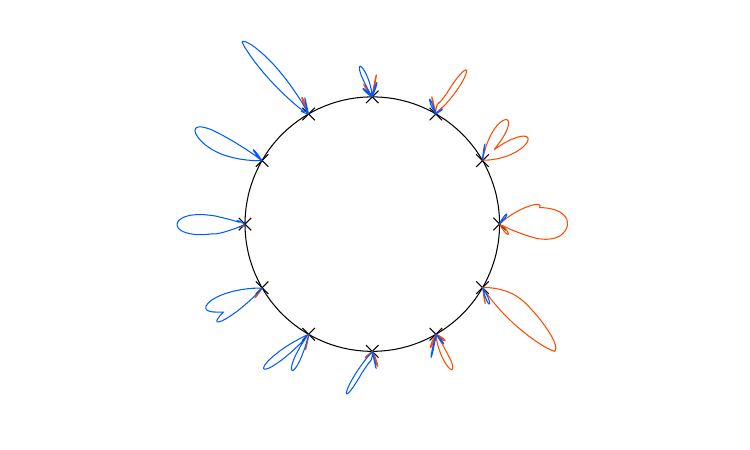}
	\end{minipage}
	\hfill
	\begin{minipage}{.48\textwidth}
		\includegraphics[width=.55\textwidth,bb=75 20 270 210]{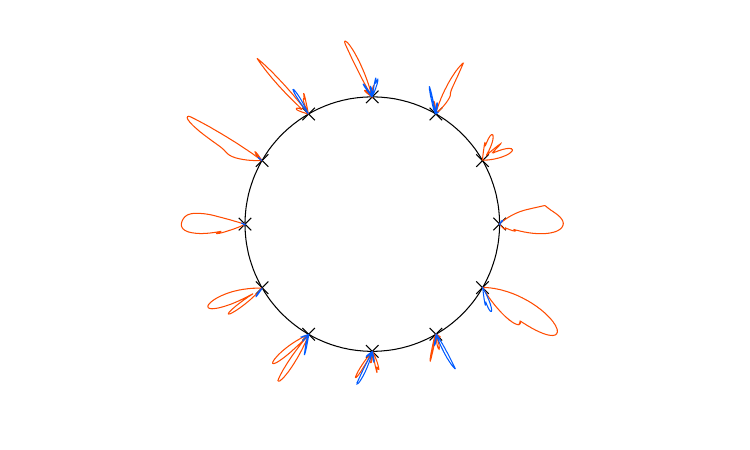}
	\end{minipage}
	\caption{\label{fig:exampleSQ:meas}
		Simulated noiseless data $I^0\bbf$ (left) and $I^1\bbf$ (right) for $\bbf$ in \eqref{eq:exampleSQ}.
		The crosses ($\times$) are some data collection points at the boundary, while the red and blue curves represent 
		$I^j\bbf(x,\btheta)$, $j=0,1$ in polar coordinates $\bigl(|I^j\bbf(x,\btheta)|,\btheta\bigr)$ centered at the respective boundary point $x\in\Gam$.
	}
\end{figure}  
\begin{figure}[h]
	\centering
	\begin{minipage}{.3\textwidth}
		\includegraphics[width=\textwidth,bb=45 0 300 215,clip]{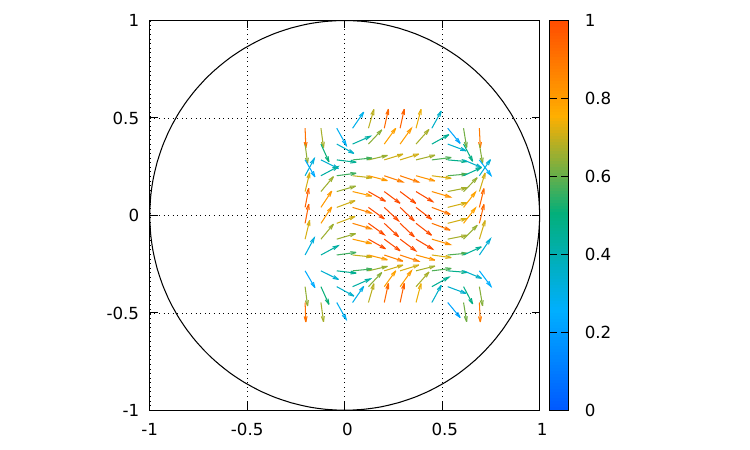}
	\end{minipage}
	\hfill
	\begin{minipage}{.33\textwidth}
		\includegraphics[width=\textwidth,bb=120 45 280 170]{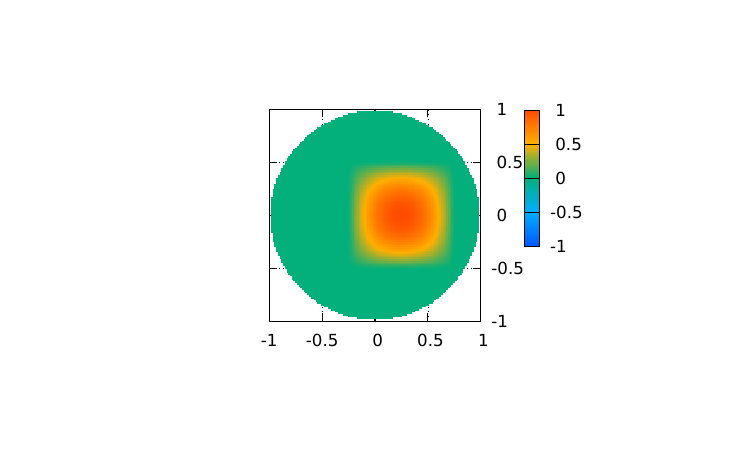}
	\end{minipage}
	\hfill
	\begin{minipage}{.33\textwidth}
		\includegraphics[width=\textwidth,bb=120 45 280 170]{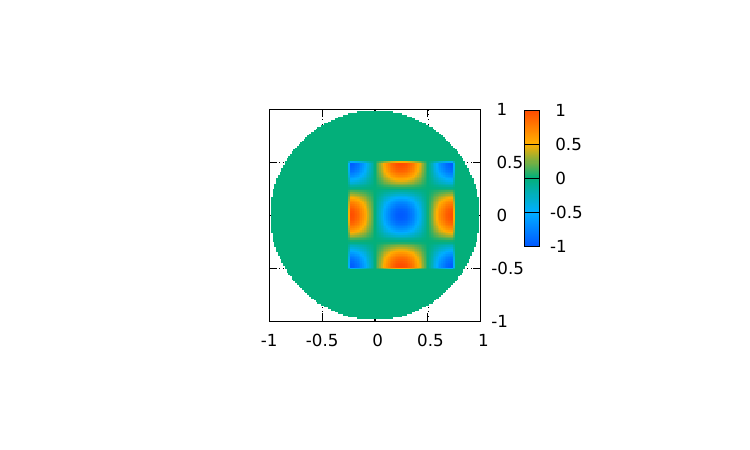}
	\end{minipage}
	\caption{\label{fig:exampleSQ:exact}Exact vector field $\bbf$ in  \eqref{eq:exampleSQ} (left),  its first component $F_1$ (middle) and its second component $F_2$ (right).}
\end{figure}
\begin{figure}[h]
	\begin{minipage}{.3\textwidth}
		\includegraphics[width=\textwidth,bb=45 0 300 215,clip]{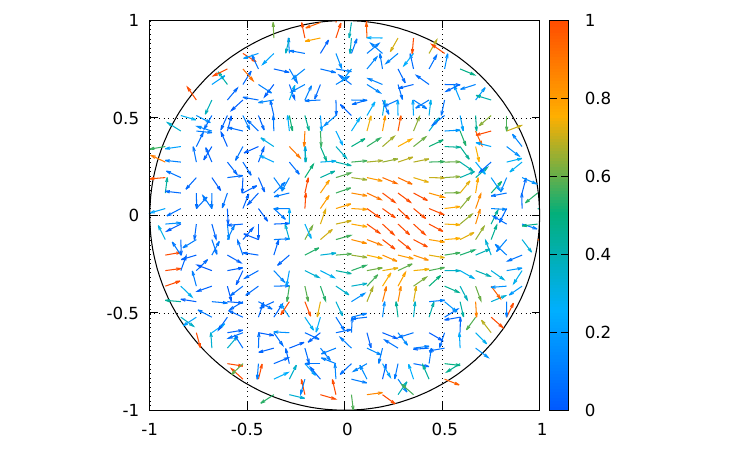}
	\end{minipage}
	\hfill
	\begin{minipage}{.33\textwidth}
		\includegraphics[width=\textwidth,bb=120 45 280 170]{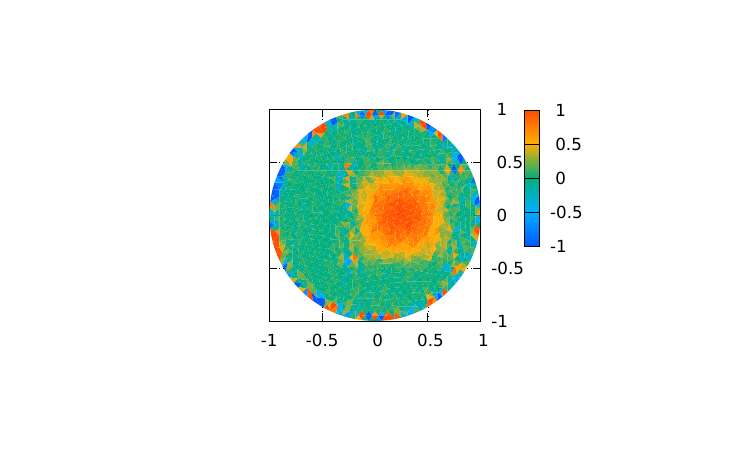}
	\end{minipage}
	\hfill
	\begin{minipage}{.33\textwidth}
		\includegraphics[width=\textwidth,bb=120 45 280 170]{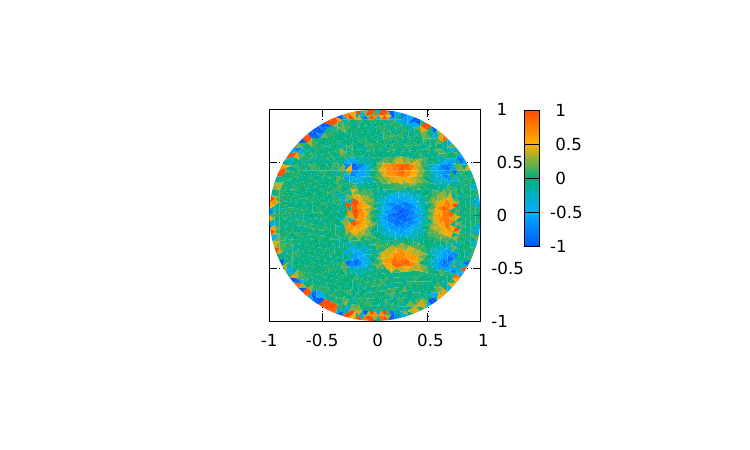}
	\end{minipage}
	\caption{\label{fig:exampleSQ}Numerically reconstructed vector field from noiseless Doppler data:
		vector field $\bbf_{0,\text{recon}} = \langle F_1, F_2 \rangle $ (left) and its components $F_1$ (middle) and $F_2$ (right).}
\end{figure}  
\end{example}

{\bf {Example~\ref{ex:SQ}(a) - Noiseless data:}}
The  vector field in \eqref{eq:exampleSQ} is given
in Figure~\ref{fig:exampleSQ:exact},  while its corresponding simulated Doppler data is shown in Figure~\ref{fig:exampleSQ:meas}. The numerically reconstructed vector field and its components  are exhibited in Figure~\ref{fig:exampleSQ}. 

{\bf {Example~\ref{ex:SQ}(b) - Noisy data:}}
We consider the same vector field as in \eqref{eq:exampleSQ}, however  the data is corrupted with additive random  errors: $5.52\%$ relative error in $I^0\bbf$, and $4.61\%$  in $I^1\bbf$. 
The numerical reconstruction results are shown in Figure~\ref{fig:exampleSQ:witherror}. 
Even though $\bbf$ has discontinuity along $\partial S$, the support of $\bbf$ can be clearly distinguished.

Fig.~\ref{fig:exampleSQ} and Fig.~\ref{fig:exampleSQ:witherror} indicate that the large error appear near the boundary.
Table~\ref{tbl:errordependency} shows this fact quantitatively. In the whole domain, the relative $L^2$ error of
reconstructed vector fields from noiseless data and noisy data are $125.5\%$ and $141.9\%$ respectively.
If we measure the errors in triangles whose centers $c=(c_1,c_2)$ locate in $|c| < 0.95$,
errors are $52.5\%$ and $61.2\%$ respectively. The number of triangles are $1{,}552$ which is $88.7\%$ of
total number of triangles ($1{,}750$) in $\Omega$, while the errors are less than half.
\begin{table}
  \caption{Errors measured in subdomains of $\Omega$ for Example~\ref{ex:SQ}}
  \label{tbl:errordependency}
  \centering
  \begin{tabular}{l|cc}
    \toprule
        & (a) Noiseless case         & (b) Random Noise case \\
    \midrule
    \midrule
    $L^2(\Omega)$ & $125.5\%$       & $141.9\%$ \\
    Number of triangles & $1{,}750$ &          \\
    \midrule
    $L^2(|x|<0.95)$ & $52.5\%$      & $61.2\%$ \\
    Number of triangles & $1{,}552$ &          \\
    \bottomrule
  \end{tabular}
\end{table}

\begin{figure}[h]
\begin{minipage}{.3\textwidth}
	\includegraphics
	[width=\textwidth,bb=45 0 300  215,clip]
	{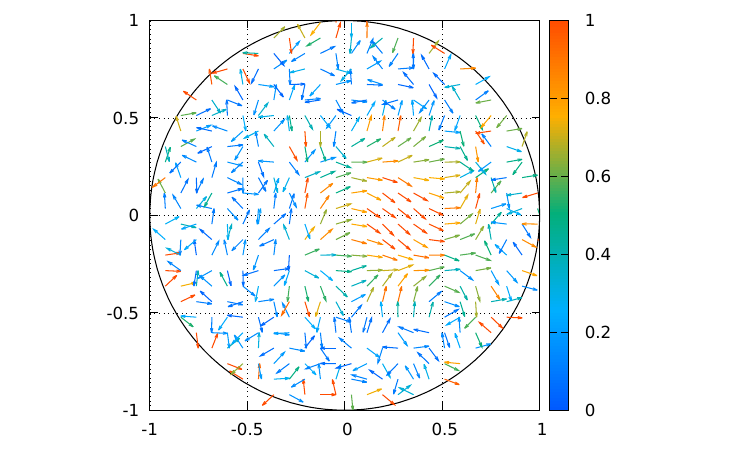}
\end{minipage}
\hfill
\begin{minipage}{.33\textwidth}
	\includegraphics[width=\textwidth,bb=120 45 280 170]{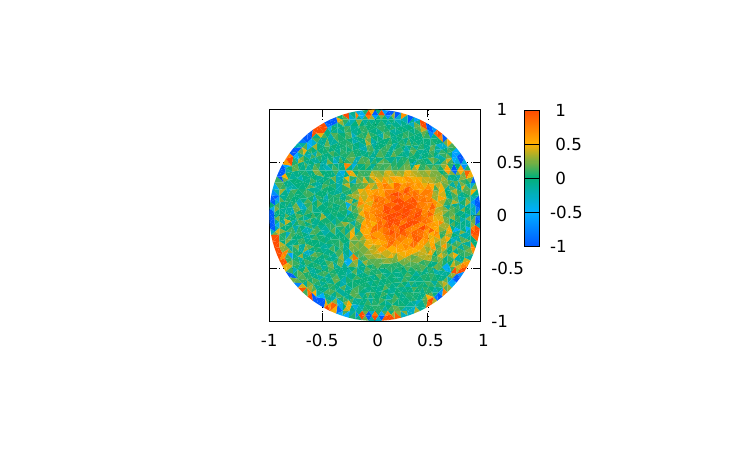}
\end{minipage}
\hfill
\begin{minipage}{.33\textwidth}
	\includegraphics[width=\textwidth,bb=120 45 280 170]{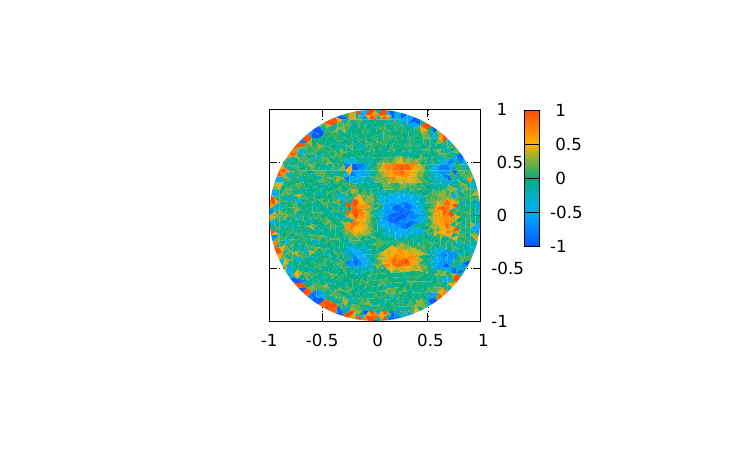}
\end{minipage}
\caption{\label{fig:exampleSQ:witherror}
  Numerical reconstruction from noisy data ($5.52\%$ relative error in $I^0\bbf$ and $4.61\%$  in $I^1\bbf$). 
}
\end{figure}  

\section{conclusion}\label{sec:conclusion}

We reduced the inversion of the  non/attenuated momenta-ray transform of symmetric real valued $m$- tensors to an inverse boundary value problem for a weakly coupled system of transport equations. In the plane, the latter is solved by an extension of Bukhgeim's theory of $A$-analytic maps. The inversion method does not require the Helmholtz decomposition and applies to sufficiently smooth planar tensors of arbitrary order. 
The stability estimates obtained here show that 
our  reconstruction method is as ill-posed as taking $m+1^{st}$ derivatives, where $m$ is the order of the tensor fields.
For the Doppler $(m=1)$ case, the method is feasible as showcased here in results  obtained in its numerical implementation.
For higher order tensors the reconstruction would require either some additional regularization, or transversal information.



\section*{Acknowledgment}
The work of H.~Fujiwara was supported by JSPS KAKENHI Grant Numbers JP20H01821, JP21H00999 and JP22K18674. The work of A.~Tamasan  was supported in part by the National Science Foundation DMS-1907097.
The work of D.~Omogbhe was supported in part by the Austrian Science Fund (FWF), Project  10.55776/P31053 and by FWF SFB 10.55776/F68 ``Tomography Across the Scales'' project F6801-N36.
The work of K.~ Sadiq  was supported in part by the FWF Project  10.55776/P31053.
For open access purposes, the author  has applied a CC BY public copyright license to any author-accepted manuscript version arising from this submission.



	\appendix



\section{Proof of Proposition \ref{prop_transEq}}\label{sec:elementary_results}

	
	From \eqref{TransportEq_u0} and \eqref{TransportEq_uk}, we note that  for $(x,\btheta)\in\OM \times \sph$ and $1 \leq k \leq m$, 
	\begin{align}\label{dt_uzero-f}
		\frac{d}{dt} \left [ e^{-\int_t^\infty a( x+s\btheta)ds} u^0(x+t\btheta, \btheta) \right]&=  e^{-\int_t^\infty a( x+s\btheta)ds}  \langle  \bbf(x+t\btheta), \btheta^m \rangle,  
		\\ \label{dt_uk-uk-1}
		\frac{d}{dt} \left [ e^{-\int_t^\infty a( x+s\btheta)ds} u^k(x+t\btheta, \btheta) \right]&=  e^{-\int_t^\infty a( x+s\btheta)ds}  u^{k-1}(x+t\btheta, \btheta).
	\end{align}
	
	For $(x,\btheta)\in\OM \times \sph$ an integration along the line through $x$ in the direction of $\btheta$ in \eqref{TransportEq_u0} together with the zero incoming condition \eqref{uk_Gam-} yield
	\begin{equation}\label{u0-g0-I0}
		\begin{aligned} 
			e^{-\int_{x \cdot \btheta}^\infty a( \Pi_{\btheta}(x)+s\btheta)ds}	u^{0}(x,\btheta)  &= \int_{-\infty}^{x \cdot \btheta}
			\frac{d}{dt} \left [ e^{-\int_t^\infty a( \Pi_{\btheta}(x)+s\btheta)ds} u^0(\Pi_{\btheta}(x)+t\btheta, \btheta) \right]
			dt \\
			&= \int_{-\infty}^{x \cdot \btheta}  e^{-\int_t^\infty a( \Pi_{\btheta}(x)+s\btheta)ds} \langle  \bbf(\Pi_{\btheta}(x)+t\btheta) , \btheta^m \rangle dt.
		\end{aligned}
	\end{equation}
	Note that $\ds \int_{x \cdot \btheta}^\infty a( \Pi_{\btheta}(x)+s\btheta)ds = \int_{0}^\infty a( x+s\btheta)ds$.
	
	Similarly, for each $1 \leq k \leq m$ a recursive integration by parts  in  \eqref{TransportEq_uk} together with \eqref{uk_Gam-} yield 
	\begin{align*}
		&e^{-\int_{x \cdot \btheta}^\infty a( \Pi_{\btheta}(x)+s\btheta)ds}	u^{k}(x,\btheta) = \int_{-\infty}^{x \cdot \btheta} 	\frac{d}{dt} \left [ e^{-\int_t^\infty a( \Pi_{\btheta}(x)+s\btheta)ds} u^k(\Pi_{\btheta}(x)+t\btheta, \btheta) \right]dt \\
		&= \int_{-\infty}^{x \cdot \btheta}  e^{-\int_t^\infty a( \Pi_{\btheta}(x)+s\btheta)ds}u^{k-1}(\Pi_{\btheta}(x)+t\btheta, \btheta) dt \\
		&=e^{-\int_{x \cdot \btheta}^\infty a( \Pi_{\btheta}(x)+s\btheta)ds} \sum_{n=1}^{k} (-1)^{n-1} \frac{(x \cdot \btheta)^n}{n !} u^{k-n}(x,\btheta)  
		\\
		&\quad+   (-1)^{k}  \int_{-\infty}^{x \cdot \btheta}\frac{t^k}{k ! } e^{-\int_t^\infty a( \Pi_{\btheta}(x)+s\btheta)ds} \langle  \bbf(\Pi_{\btheta}(x)+t\btheta) , \btheta^m \rangle dt,
	\end{align*}where in the last equality we use \eqref{u0-g0-I0}.
	Thus,  multiplying both sides of the above equation with $\ds e^{\int_{x \cdot \btheta}^\infty a( \Pi_{\btheta}(x)+s\btheta)ds}$ 
	yields 
	\begin{equation}\label{uk-gk-Ik}
		\begin{aligned}
			u^{k}(x,\btheta) &= \sum_{n=1}^{k} (-1)^{n-1} \frac{(x \cdot \btheta)^n}{n !} u^{k-n}(x,\btheta)
			+ (-1)^{k} 	 \int_{-\infty}^{x \cdot \btheta}\frac{t^k}{k ! } e^{-\int_t^{x\cdot \btheta} a( \Pi_{\btheta}(x)+s\btheta)ds} \langle  \bbf(\Pi_{\btheta}(x)+t\btheta) , \btheta^m \rangle dt,
		\end{aligned}
	\end{equation}

	Since $\ds \bbf( x + (t - x \cdot \btheta) \btheta) = \bzero$ for every $(x,\btheta)\in\Gam_+$ and $t>x\cdot\btheta$, 
		\begin{align} \nonumber 
			\int_{-\infty}^{x \cdot \btheta}t^k e^{-\int_t^{x \cdot \btheta} a( \Pi_{\btheta}(x)+s\btheta)ds}\langle  \bbf(\Pi_{\btheta}(x)+t\btheta) , \btheta^m \rangle dt &=
			\int_{-\infty}^{\infty}t^k e^{-\int_t^{x \cdot \btheta} a( \Pi_{\btheta}(x)+s\btheta)ds} \langle  \bbf(\Pi_{\btheta}(x)+t\btheta) , \btheta^m \rangle dt  \\ \label{suppf_OM} 
			&=I_a^k\bbf (x, \btheta).
		\end{align}
	The relations \eqref{eq:gk-Ik} now follow from \eqref{uk-gk-Ik}, \eqref{u0-g0-I0}, and \eqref{suppf_OM}. 
	
	Since $\ds \bbf \in H^{s}_0(\mathbf{S}^m; \OM)$, $s \geq 1$ and   $a\in C^{s,\mu}(\ol\OM), \mu >1/2$, the solution $u^0 $ given by \eqref{u0-g0-I0} preserves the regularity and $\ds u^0 \in H^{s}(\Omega\times\sph)$. Moreover, by \eqref{TransportEq_uk} and \eqref{uk-gk-Ik}, $\ds u^k \in H^{s}(\Omega\times\sph), s \geq 1$ for $1\leq k \leq m$.
	\qed

 
\section{An explicit Pompeiu formula for  $L^2$-analytic maps } 
\label{sec:L2map}

In this section  we derive a Pompeiu type formula corresponding to $A$-analytic maps. 
The domain $\OM$ is strictly convex but need not be a disc.

Bukhgeim's original  theory in \cite{bukhgeimBook} shows that solutions (called $L^2$-analytic) of the  
homogenous Beltrami-like equation 
\begin{align}\label{beltrami}
	\dba\bv(z) + L^2 \del\bv(z) = \bzero,\quad z\in \OM,
\end{align}
satisfy a Cauchy-like integral formula,
\begin{align}\label{Analytic}
	\bv (z) = \B [\bv \lvert_{\Gam}](z), \quad  z\in\OM,
\end{align} where 
$\B$ is defined component-wise for $n\geq 0$ by
\begin{equation}\label{BukhgeimCauchyFormula}
	\begin{aligned} 
		(\B \bv)_{-n}(z) &:= \frac{1}{2\pi \i} \int_{\Gam}
		\frac{ v_{-n}(\zeta)}{\zeta-z}d\zeta  
		\\ 
		&\qquad 
		+ \frac{1}{2\pi \i}\int_{\Gam} \left \{ \frac{d\zeta}{\zeta-z}-\frac{d \ol{\zeta}}{\ol{\zeta}-\ol{z}} \right \} \sum_{j=1}^{\infty}  
		v_{-n-2j}(\zeta)
		\left( \frac{\ol{\zeta}-\ol{z}}{\zeta-z} \right) ^{j}, \, z\in\OM.
	\end{aligned}
\end{equation}

For the inhomogeneous Bukhgeim-Beltrami equation 
\begin{align}\label{BukhBeltsimple1}
	\dba\bv +L^2 \del\bv = \bw,
\end{align} 
we introduce here a Pompeiu like operator $\BT$  defined component-wise for $n\geq 0$ by
\begin{align}\label{T_Formula}
	(\BT \bw)_{-n}(z) :=  -\frac{1}{ \pi } \sum_{j=0}^{\infty} \int_{\OM}    w_{-n-2j} (\zeta)  \frac{1}{\zeta-z} \left( \frac{\ol{\zeta}-\ol{z}}{\zeta-z} \right) ^{j}d\xi d\eta, \quad \zeta = \xi +\i \eta, \quad z \in \OM.
\end{align}

	\begin{prop}
		\label{prop_Bukhgeimpompeiu}
Let $\Omega$ be bounded convex domain with $C^1$ boundary,  $\B$ and $\BT$ be the operators in  \eqref{BukhgeimCauchyFormula}, respectively,
\eqref{T_Formula}, and $\bw \in  C(\ol\OM; l^1)$. 
  	 If $\bv \in C^1(\OM; l^1)\cap C(\ol\OM; l^1)$ solves \eqref{BukhBeltsimple1}, then 
	\begin{align} \label{BP_Formula_vk}
		\bv(z) &= 
		\B [\bv \lvert_{\Gam}](z) +	(\BT \bw)(z), 	\quad z \in \OM.
	\end{align}
	\end{prop}

\begin{proof}

	For  $n \geq 0$, consider 
	\begin{align}\label{sigma_n}
		\displaystyle \sigma_{-n}(z,\varphi) = \sum_{j=0}^\infty v_{-n-2j}(z)e^{-\i (n+2j)\varphi},
		\end{align}
	as in \cite{finch}. It is easy to see that  $\sigma_{-n} \in C^1(\Omega)\cap C(\overline{\Omega}).$ 
	
	Let $z \in \Omega$ be arbitrarily fixed.
	  We parametrize  $\ol \OM$ in polar coordinates
	$$\zeta (\varphi) = z+te^{\i \varphi}, \quad 0 \leq t \leq l(\varphi):=\dist(z, \Gam), \quad 0 \leq \phi \leq 2 \pi.$$ Then 
	\begin{align} 		\label{eq:measure}
		e^{-2\i\varphi} = \frac{\overline{\zeta}- \ol{z}}{\zeta -z}, \qquad d\varphi =\frac{1}{2\i}\Bigg(\frac{1}{\zeta-z}d\zeta- \frac{1}{\overline{\zeta}-\overline{z}}d\overline{\zeta} \Bigg).
	\end{align} 

	The  equation \eqref{BukhBeltsimple1} 
written in component-wise form is 
\begin{align} \label{Beltrami_ukEq1}
	&\ol{\del} v_{-n}(z) + \del v_{-n-2}(z) = w_{-n}(z), \quad  z\in\OM, \; n \in \BZ. 
\end{align}
	
	For $n\geq 0,$ evaluate for each $\varphi \in [0, 2 \pi]$
	\begin{align*}
		\displaystyle \sigma_{-n}(\zeta,\varphi)-\sigma_{-n}(z,\varphi) & = \int_{0}^l \frac{\partial \sigma_{-n}}{\partial t}(z + t e^{\i\varphi},\varphi)dt
		= \int_{0}^l \Big(\frac{\partial \sigma_{-n}}{\partial \overline{z}}e^{-\i\varphi} +\frac{\partial \sigma_{-n}}{\partial z}e^{\i\varphi} \Big)dt 
		\\ \numberthis \label{eq:sigma_n}
		& = \int_{0}^l \frac{\partial v_{-n}}{\partial z}e^{-\i(n-1)\varphi}dt +
		\int_{0}^l \sum_{j=0}^\infty
		\Big(\dba v_{-n-2j} + \del v_{-n-2j-2}\Big) e^{-\i(n+2j+1)\varphi} dt.
	\end{align*} 
	
	From \eqref{sigma_n},  $v_{-n}(z)$ is the $n$-th Fourier coefficient of $\sigma_{-n}(z, \cdot)$. Thus, 
		\begin{align*}
			\displaystyle v_{-n}(z) & = \frac{1}{2\pi} \int_{0}^{2\pi}\sigma_{-n}(z,\varphi)e^{\i n\varphi} d \varphi
			\\& = \frac{1}{2\pi} \int_{0}^{2\pi}\sigma_{-n}(\zeta,\varphi)e^{\i n\varphi} d \varphi - \frac{1}{2\pi} \int_{0}^{2\pi} \int_{0}^{l(\varphi)} 
			\del v_{-n} 
			 \frac{1}{te^{-\i\varphi}}t dtd \varphi
			\\&\qquad - \frac{1}{2\pi} \int_{0}^{2\pi} \int_{0}^{l(\varphi)}
			\sum_{j=0}^\infty 
			\Big(\dba v_{-n-2j} + \del v_{-n-2j-2}\Big)
			e^{-2\i j\varphi} \frac{1}{te^{\i\varphi}}tdtd \varphi
			\\ \numberthis \label{eq:BPformula}
			& = \frac{1}{2\pi} \int_{0}^{2\pi}\sum_{j=0}^\infty v_{-n-2j}(\zeta)e^{-2\i j\varphi} d \varphi - \frac{1}{2\pi} \int_{\Omega}
			\del v_{-n}
			\frac{1}{\overline{\zeta}-\overline{z}}d\xi d\eta
			\\&\qquad 
			- \frac{1}{2\pi} \int_{\Omega}
			\sum_{j=0}^\infty  \Big(\dba v_{-n-2j} + \del v_{-n-2j-2}\Big)
			e^{-2\i j\varphi}
			\frac{1}{\zeta- z} \, d\xi d\eta, \quad \zeta = \xi +\i \eta, 
		\end{align*}
	where the second equality uses \eqref{eq:sigma_n}. 	

	By the conjugate form of the Cauchy-Pompeiu formula (e.g. see \cite{vekua_book62}), we have
	\begin{align} 		\label{eq:conjCP}
		\frac{1}{2\pi} \int_{\Omega} 
			\dba v_{-n}
		\frac{1}{\overline{\zeta- z}}d\xi d\eta =\frac{1}{2}v_{-n}(z) + \frac{1}{4\pi \i}\int_{\partial \Omega} v_{-n}(\zeta)\frac{1}{\overline{\zeta}-\overline{z}} d \zeta.
	\end{align}
	Substituting  \eqref{Beltrami_ukEq1}, \eqref{eq:measure}, and \eqref{eq:conjCP} into \eqref{eq:BPformula} yields  
	\begin{equation} \label{eq:u0uk_BPformula1}
		\begin{aligned}
			v_{-n}(z)  &= \frac{1}{2\pi \i} \int_{\Gam}
			\frac{ v_{-n}(\zeta)}{\zeta-z}d\zeta  + \frac{1}{2\pi \i}\int_{\Gam} \left \{ \frac{d\zeta}{\zeta-z}-\frac{d \ol{\zeta}}{\ol{\zeta}-\ol{z}} \right \} \sum_{j=1}^{\infty}  
			v_{-n-2j}(\zeta) 	\left( \frac{\ol{\zeta}-\ol{z}}{\zeta-z} \right) ^{j} \\
			& \qquad  -\frac{1}{ \pi } \sum_{j=0}^{\infty} \int_{\OM}    w_{-n-2j}(\zeta)  \frac{1}{\zeta-z} \left( \frac{\ol{\zeta}-\ol{z}}{\zeta-z} \right) ^{j}d\xi d\eta,  \quad \zeta = \xi +\i \eta.
		\end{aligned}
	\end{equation}

\end{proof}	




\begin{thebibliography}{99}
		
		\bibitem{abbeyetal}B. ~Abbey, S. ~ Zhang, M.~ Xie, X. ~Song, A. ~Korsunsky, \textit{Neutron
		strain tomography using bragg-edge transmission}, International journal
		of materials research \textbf{103 (2)} (2012),  234--241.
		
		\bibitem{aben79} H. ~Aben, Integrated Photoelasticity, McGraw-Hill, New York, 1979.
		
		
		\bibitem{abhishekMishra} A.~ Abhishek and R.~K.~Mishra, \textit{Support theorems and an injectivity result for integral moments of a symmetric $m$-tensor field}, J. Fourier Anal. Appl., \textbf{25 (4)} (2019), 1487-1512.
		
%

      \bibitem{andersson} F.~ Andersson, \textit{ The Doppler moment transform in Doppler tomography}, 
		Inverse Problems \textbf{21} (2005), 1249--1274.
		
		\bibitem{ABK} E.~V.~Arbuzov, A.~L.~Bukhgeim and S.~G.~Kazantsev, \textit{Two-dimensional tomography problems and the theory of A-analytic functions}, Siberian Adv. Math., \textbf{8} (1998), 1--20.
		
		
\bibitem{bal04} G.~Bal, \textit{On the attenuated Radon transform with full and partial measurements}, Inverse Problems \textbf{20} (2004), 399--418.
		
		
		
		\bibitem{braunHauk} H. Braun and A. Hauk, \textit{Tomographic reconstruction of vector
			fields}, IEEE Transactions on signal processing \textbf{39} (1991), 464--471.
		
		
		\bibitem{bukhgeimBook} A.~L.~Bukhgeim, \textit{Inversion Formulas in Inverse Problems}, chapter in Linear Operators and Ill-Posed Problems by M. M. Lavrentiev and L. Ya. Savalev, Plenum, New York, 1995.
		
		
		%
		
		

		\bibitem{derevtsov23} E.~Derevtsov, \textit{ Ray Transforms of the Moments of Planar Tensor Fields},  J. Appl. Ind. Math., \textbf{17} (2023), 521--534.
		
		\bibitem{derevtsovSvetov15} E.~Derevtsov and I.~Svetov, \textit{Tomography of tensor fields in the plane}, Eurasian J. Math. Comput. Appl., \textbf{3(2)} (2015), 24--68.
		
		\bibitem{derevtsovetal21} E.~Derevtsov, Y.~Volkov and T.~Schuster, \textit{Generalized attenuated ray transforms and their integral angular moments}, Appl. Math. and Comput., \textbf{409} (2021), 125494.

\bibitem{denisiuk}  A.~Denisiuk, \textit{Iterative inversion of the tensor momentum x-ray transform}, Inverse Problems  \textbf{39 (10)}, (2023),  105002.
		
\bibitem{desaiLionheart16} N.~Desai  and W.~Lionheart,  \textit{An explicit reconstruction algorithm for the transverse ray transform of a second rank tensor field from three axis data},   Inverse Problems, \textbf{32(11)}, (2016), 115009.
		 
		
		
		\bibitem{finch} D.~V.~Finch, \textit{The attenuated x-ray transform: recent developments}, in Inside out: inverse problems and applications, Math. Sci. Res. Inst. Publ., \textbf{47}, Cambridge Univ. Press, Cambridge, 2003, 47--66.
		
		\bibitem{fujiwaraSadiqTamasan19} H.~Fujiwara, K.~Sadiq and A.~Tamasan, \textit{A Fourier approach to the inverse source problem in an absorbing and anisotropic scattering medium}, Inverse Problems {\bf 36(1)}:015005 (2019).
		
		\bibitem{fujiwaraSadiqTamasan20} H.~Fujiwara, K.~Sadiq and A.~Tamasan, \textit{Numerical reconstruction of radiative sources in an absorbing and non-diffusing scattering medium in two dimensions}, 
		SIAM J. Imaging Sci., {\bf 13(1)} (2020), 535--555. 
		
		
		
		
		

\bibitem{haideretal} S.~ Haider, A.~Hrbek and Y.~Xu, 
\textit{ Magneto-acousto-electrical tomography: a potential method for imaging current density and electrical impedance} Physiol. Meas. \textbf{29} (2008) 41--50.
		
		
		

\bibitem{hendriksetal} J. ~Hendriks, A. ~Gregg, C. ~Wensrich, A.~Tremsin, T.~Shinohara, M. ~Meylan, 
E. ~Kisi, V.~Luzin, O.~Kirsten, \textit{Bragg-edge elastic strain tomography for in situ systems from energy-resolved neutron transmission imaging}, Physical Review Materials \textbf{1 (5)} (2017) 053802.


	\bibitem{holmanStefanov} S.~Holman and P.~Stefanov, \textit{The weighted Doppler transform}, Inverse Probl. Imaging, \textbf{4} (2010), 111--130.
		
		
		
		
		
		       \bibitem{kazantsevBukhgeimJr07}  S.~G.~Kazantsev and A.~A.~Bukhgeim, \textit{Inversion of the scalar and vector attenuated X-ray transforms in a unit disc}, J. Inverse Ill-Posed Probl., \textbf{15} (2007), 735--765.
		
		\bibitem{venkeetal19} V.~P.~Krishnan, R.~Manna, S.~K.~ Sahoo, and V.~A.~ Sharafutdinov, \textit{Momentum ray transforms}, Inverse Problems Imaging \textbf{3 (13)} (2019), 679--701.
		
		
		 \bibitem{venkeMishraMonard19} V.~P.~Krishnan, R.~Mishra and F.~Monard  \textit{On solenoidal-injective and injective ray transforms of tensor fields on surfaces},  J. Inverse Ill-Posed Problems \textbf{27 (4)} (2019), 527--538.


\bibitem{kunyanskyetal23} L.~Kunyansky, E.~McDugald, and B.~Shearer,  \textit{Weighted Radon transforms of vector fields, with applications to magnetoacoustoelectric tomography},	Inverse Problems \textbf{39 (6)} (2023), 065014.




		\bibitem{louis}  A.~Louis, \textit{Inversion formulae for ray transforms in vector and tensor tomography}, Inverse Problems {\bf 38}:065008 (2022).

		
		
		
		\bibitem{mishra} R.~K.~Mishra, \textit{Full reconstruction of a vector field from restricted Doppler and first integral moment transforms in $\BR^n$},  J. Inverse Ill-Posed Problems \textbf{28} (2019), 173--184.
		
		\bibitem{mishraSahoo} R.~K.~Mishra and S.~K.~Sahoo, \textit{ Injectivity and range description of integral moment transforms over $m$-tensor fields in $\BR^n$}, SIAM J. Math. Anal., \textbf{53 (1)}, (2021), 253-278.
		
		
		
		
		
		
		
		
		
		
		\bibitem{nattererBook} F.~Natterer, {The mathematics of computerized tomography}, Wiley, New York, 1986.
		
		
		
			\bibitem{norton88} S. J. Norton, \textit{Tomographic reconstruction of 2-D vector fields: application to flow imaging}, Geophysical Journal \textbf{97 (1)} (1989), 161--168.

		
		

		
		
		
		

		
\bibitem{palamodov09} V.~Palamodov, \textit{Reconstruction of a differential form from doppler transform}, SIAM Journal on Mathematical Analysis, \textbf{41(4)} (2009), 1713--1720.
		
		
		
		
		
		
		
		
		
%
		\bibitem{paternainSaloUhlmann14} G. P. Paternain, M. Salo, and G. Uhlmann, \textit{Tensor Tomography: Progress and Challenges}, Chin. Ann. Math. Ser. B., \textbf {35(3)} (2014), 399--428.
		
		
		
		
		
		
		
\bibitem{pestovUhlmann04} L. Pestov and G. Uhlmann, \textit{On characterization of the range and inversion formulas for the geodesic X-ray transform}, Int. Math. Res. Not., \textbf{80} (2004), 4331--4347.
		
		
		
		
		\bibitem{sadiqTamasan01} K.~Sadiq and A.~Tamasan, \textit{On the range of the attenuated Radon transform in strictly convex sets}, Trans. Amer. Math. Soc., \textbf{367(8)} (2015), 5375--5398.
		
		\bibitem{sadiqTamasan02} K.~Sadiq and A.~Tamasan, \textit{On the range characterization of the two dimensional attenuated Doppler transform},  SIAM J. Math. Anal., \textbf{47(3)} (2015), 2001--2021.
		
		
	\bibitem{sadiqTamasan23} K.~Sadiq, A.~Tamasan, \textit{ On the range of the $X$-ray transform of symmetric tensors compactly supported in the plane},  Inverse Probl. Imaging  \textbf{17(3)} (2023), 660--685.
		
		\bibitem{sadiqScherzerTamasan} K.~Sadiq, O.~Scherzer, and A.~Tamasan, \textit{On the X-ray transform of planar symmetric 2-tensors}, J. Math. Anal. Appl., \textbf{442(1)} (2016),  31--49.
		
		
		

\bibitem{schuster08} T. Schuster, \textit{20 years of imaging in vector field tomography: a review}. In Y. Censor,
		M. Jiang, A.K. Louis (Eds.), \textit{Mathematical Methods in Biomedical Imaging and Intensity-Modulated Radiation Therapy (IMRT)}, in: Publications of the Scuola Normale Superiore, CRM \textbf{7} (2008) 389--424.
		
		
		\bibitem{sharafutdinov_book94} V. A. Sharafutdinov, \textit{Integral geometry of tensor fields}, VSP, Utrecht, 1994.
	
\bibitem{sharafutdinov17} V.~A.~Sharafutdinov, \textit{The Reshetnyak formula and Natterer stability estimates in tensor
	tomography}, Inverse Problems, \textbf{33 (2)} (2017), 025002.	
			
		
		
		\bibitem {sparSLP95} G. Sparr, K. Str{\aa}hl\'en, K. Lindstr\"om, and H. W. Persson, \textit{Doppler tomography for vector fields,} Inverse Problems, \textbf{11} (1995), 1051--1061.
		
		
		 \bibitem{stefanovUhlmann98} P.~Stefanov, and G.~Uhlmann, \textit{Rigidity for metrics with same lengths of 	geodesics}, Math. Res. Lett. , \textbf{5}, (1998)  83--96.
		
		\bibitem{stefanovUhlmannetal} P.~Stefanov, G.~Uhlmann, A.~Vasy, and H.~Zhou,  \textit{Travel time tomograph}, Acta Math. Sin. (Engl. Ser.), \textbf{35(6)} (2019), 1085--1114.
		
		
		
		
		
		\bibitem{tamasan07} A.~Tamasan, \textit{Tomographic reconstruction of vector fields in variable background media}, Inverse Problems \textbf{23} (2007), 2197--2205.



	
		\bibitem{vekua_book62} I.~N.~Vekua, \textit{Generalized Analytic Functions}, Pergamon Press Ltd.\ 1962.
	
	
		
		
		
		
		
		%
		
		
		
		
		
	\end{thebibliography}
\end{document}